\documentclass[12pt,a4paper]{amsart}

\usepackage{fullpage}
\usepackage[pdftex]{epsfig}

\usepackage{amssymb}  
\usepackage{latexsym} 
\usepackage{comment}
\usepackage{color}
\usepackage{enumerate}
\usepackage{amsthm}
\usepackage[all]{xy}
\usepackage{colonequals}

\usepackage{tikz}
\usetikzlibrary{matrix,arrows,automata,snakes}
\usepackage[textwidth=20mm,textsize=small]{todonotes}
\usepackage{charter,euler} 
\DeclareSymbolFont{bchoperators}{T1}{bch}{m}{n}
\SetSymbolFont{bchoperators}{bold}{T1}{bch}{b}{n}
\makeatletter
\renewcommand{\operator@font}{\mathgroup\symbchoperators}
\makeatother

\usepackage{titlesec} 
\titleformat{\section}{\normalfont\bfseries\filcenter}{\thesection}{1em}{}
\titleformat{\subsection}{\normalfont\bfseries}{\thesubsection}{1em}{}
\titleformat{\subsubsection}{\normalfont\bfseries}{\thesubsubsection}{1em}{}

\numberwithin{equation}{section}

\definecolor{darkgreen}{rgb}{0,0.5,0}
\definecolor{rem}{rgb}{0.8,0,0}
\definecolor{new}{rgb}{0.7,0,0.6}
\definecolor{reply}{rgb}{0,0,0.8}
\usepackage[
        colorlinks, citecolor=darkgreen,
        backref,
        pdfauthor={Maarten Derickx, Michael Stoll},
        pdftitle={Prime order torsion on elliptic curves over number fields},
]{hyperref}
\usepackage[capitalize]{cleveref}
\usepackage[alphabetic,backrefs,lite]{amsrefs} 

\newtheorem{theorem}{Theorem}[section]
\newtheorem{lemma}[theorem]{Lemma}
\newtheorem{proposition}[theorem]{Proposition}

\newtheorem{corollary}[theorem]{Corollary}
\newtheorem{conjecture}[theorem]{Conjecture}

\theoremstyle{definition}
\newtheorem{definition}[theorem]{Definition}

\newtheorem{question}[theorem]{Question}

\theoremstyle{remark}
\newtheorem{remark}[theorem]{Remark}

\crefname{conjecture}{Conjecture}{Conjectures}
\Crefname{conjecture}{Conjecture}{Conjectures}

\newcommand{\set}[1]{\left\lbrace #1 \right\rbrace}
\newcommand{\diamondop}[1]{\langle #1 \rangle} 
\newcommand{\field}[1]{\mathbb{#1}}  
\newcommand{\Q}{\field{Q}} 
\newcommand{\Qbar}{\bar{\Q}}
\newcommand{\C}{\field{C}} 
\newcommand{\Z}{\field{Z}} 
\newcommand{\F}{\field{F}} 
\newcommand{\T}{\field{T}} 
\renewcommand{\P}{\field{P}}

\newcommand{\calO}{\mathcal{O}}

\newcommand{\bfe}{\mathbf{e}}

\newcommand{\HH}{\mathfrak{H}} 

\newcommand{\eps}{\varepsilon}

\newcommand{\floor}[1]{\left\lfloor #1 \right\rfloor}

\DeclareMathOperator{\Aut}{Aut}
\DeclareMathOperator{\CM}{CM}

\DeclareMathOperator{\disc}{disc}
\DeclareMathOperator{\End}{End}

\DeclareMathOperator{\GL}{GL}
\DeclareMathOperator{\gon}{gon}

\DeclareMathOperator{\new}{new}

\DeclareMathOperator{\ord}{ord}
\DeclareMathOperator{\perdim}{perdim}

\DeclareMathOperator{\primes}{Primes}

\DeclareMathOperator{\sdim}{strdim}

\DeclareMathOperator{\str}{str}

\DeclareMathOperator{\tors}{tors}
\DeclareMathOperator{\Tr}{Tr}

\DeclareMathOperator{\SSS}{\mathbb S}
\DeclareMathOperator{\MMM}{\mathbb{M}}

\newcommand{\Snew}{S_{\new}}

\SelectTips{eu}{} 

\begin{document}

\author{Maarten Derickx}
\address{Department of Mathematics,
	University of Zagreb,
  Bijeni\v{c}ka Cesta 30,
  10000 Zagreb, Croatia}
\email{maarten@mderickx.nl}

\author{Michael Stoll}
\address{Mathematisches Institut,
	Universität Bayreuth,
	95440 Bayreuth, Germany.}
\email{Michael.Stoll@uni-bayreuth.de}

\title[Prime order torsion on elliptic curves]%
{Prime order torsion on elliptic curves over number fields \\
Part I: Asymptotics}

\date{\today}

\begin{abstract}
	We study the asymptotics of the set~$S(d)$ of possible prime orders of $K$-rational points
	on elliptic curves over number fields~$K$ of degree~$d$ as $d$ tends to infinity.
	Assuming some conjectures on the sparsity of newforms of weight~$2$ and
	prime level with unexpectedly high analytic rank, we show that $\max S(d) \le 3d + 1$
	for sufficiently large even~$d$ and $\max S(d) = o(d)$ for odd~$d$.
	This gives a significant (but conjectural) improvement on the currently best known upper bound
	$\max S(d) \le (3^{d/2} + 1)^2$ due to~Oesterlé.
\end{abstract}

\keywords{Elliptic curve, torsion point, number field}

\subjclass[2020]{Primary 11G05, 11G18, Secondary 14G05, 14G25, 14G35}

\maketitle


\section{Introduction}

Let $K$ be an algebraic number field and let $E$ be an elliptic curve
over~$K$. Then by work of Mordell~\cite{Mordell} and Weil~\cite{Weil}
the group~$E(K)$ of $K$-rational points on~$E$
is a finitely generated abelian group;
in particular, its torsion subgroup~$E(K)_{\tors}$ is a finite abelian
group (this can also be shown in a
number of more elementary ways than relying on the Mordell-Weil theorem),
and one can ask which finite abelian groups can occur as the
torsion subgroup of~$E(K)$ for some elliptic curve over some number
field~$K$ of some fixed degree~$d$.

For $K = \Q$ (equivalently, $d = 1$), Mazur~\cites{mazur1,mazur2}
famously proved that there are only finitely many possibilities
for the torsion subgroup and thereby confirmed that the conjectured list
is complete. Later, Merel~\cite{merel} extended this
by showing that for any given degree~$d$, there are only finitely
many possibilities for~$E(K)_{\tors}$ when $[K : \Q] = d$.
These have been determined explicitly for $d = 2$ by Kamienny~\cite{kamienny2}
building on work by Kenku and Momose~\cite{KenkuMomose}, for $d = 3$
by Derickx, Etropolski, van Hoeij, Morrow and Zureick-Brown~\cite{DEvHMZ}
building on previous work of Jeon, Kim and Schweizer~\cite{JKS}
and Bruin and Najman~\cite{bruin-najman}, and very recently
also for $d = 4$ by Derickx and Najman~\cite{derickx-najman}.

One key step in these finiteness results is to show that there
are only finitely many prime numbers~$p$ that can divide the order
of~$E(K)_{\tors}$, i.e., can occur as the order of an element
of~$E(K)$, for $K$ of degree~$d$. This motivates the following
definition (following~\cite{kamiennymazur}).

\begin{definition}
  Let $d \ge 1$ be an integer. Then we define $S(d)$ to be the set
  of all prime numbers~$p$ such that there exists a number field~$K$
  of degree~$d$, an elliptic curve~$E$ over~$K$ and a point $P \in E(K)$
  such that $P$ has order~$p$.

  Following~\cite{DKSS}, we write $\primes(x)$ for the set of all prime numbers~$p$
  such that $p \leq x$.
\end{definition}

The following values of~$S(d)$ are known.

\begin{theorem} \label[theorem]{thm:known}
  \[ \renewcommand{\arraystretch}{1.2}
    \begin{array}{rcl@{\qquad}l}
      S(d) & \subseteq &\primes((3^{d/2}+1)^2)& \text{\cites{oesterle,DKSS}} \\
      S(1) & = &\primes(7)  & \text{\cites{mazur1,mazur2}}, \\
      S(2) & = &\primes(13) & \text{\cite{kamienny2}}, \\
      S(3) & = &\primes(13) & \text{\cites{parent2,parent3}}, \\
      S(4) & = &\primes(17) & \text{\cite{DKSS}}, \\
      S(5) & = &\primes(19) & \text{\cite{DKSS}}, \\
      S(6) & = &\primes(19) \cup \set{37} & \text{\cite{DKSS}}, \\
      S(7) & = &\primes(23) & \text{\cite{DKSS}, \quad and} \\
      S(8) & = &\primes(23) & \text{\cite{Khawaja}.}
    \end{array}
  \]
\end{theorem}
The bound $S(d) \subseteq \primes((3^{d/2}+1)^2)$ mentioned above is currently
the best known asymptotic  bound on $S(d)$ and is due to Joseph Oesterl\'e~\cites{oesterle}.
His notes containing the original proof of this are not publicly available,
but a proof of his bound (closely following these notes) can be found in Section~6 of~\cite{DKSS}.
The recent determination of~$S(8)$ by Khawaja follows the approach
taken in~\cite{DKSS}. In a later part of this series, we will give an
alternative proof that requires less computation.

It is much easier to determine the set~$S'(d)$ of primes~$p$ such that
there are \emph{infinitely many} elliptic curves~$E$ over number fields~$K$
of degree~$d$ with distinct $j$-invariants that have a $K$-point of
order~$p$. This is largely a question about the gonality of the modular
curve~$X_1(p)$. The following is known.

\begin{proposition}
  \begin{align*}
    S'(1) &= \primes(7), &
    S'(2) &= \primes(13), &
    S'(3) &= \primes(13), &
    S'(4) &= \primes(17), \\
    S'(5) &= \primes(19), &
    S'(6) &= \primes(19), &
    S'(7) &= \primes(23), &
    S'(8) &= \primes(23).
  \end{align*}
\end{proposition}

For $d = 1, 2, 3, 4$, this is shown in~\cites{mazur1,kamienny2,JKL1,JKL2},
respectively; for $5 \le d \le 8$, this follows
from~\cite{derickx_hoeij}*{Thm.~3}.

The gonality of~$X_1(p)$ grows like~$p^2$~\cite{abramovich}; this
implies that $S'(d) \subseteq \primes\bigl(O(\sqrt{d})\bigr)$;
see~\cref{prop:Sprime} below. On the other
hand, denoting by $S_{\CM}(d)$ the set of primes that can occur as
orders of points on elliptic curves with complex multiplication
over a number field of degree~$d$, the results of~\cite{CCS} show
that $S_{\CM}(d) \subseteq \primes\bigl(O(d)\bigr)$ and that
$3d+1 \in S_{\CM}(d)$ when $3d+1$ is prime\footnote{Let $p = 3d+1$. There is a
pair of quadratic points defined over~$\Q(\sqrt{-3})$ with $j$-invariant
zero on~$X_0(p)$. The set-theoretic preimage gives a Galois orbit of points
of degree $2 \cdot \tfrac{p-1}{2} \cdot \tfrac{1}{3} = d$ on~$X_1(p)$,
since the covering $X_1(p) \to X_0(p)$ ramifies with index~$3$ above the
points with $j$-invariant zero.}.
Since clearly $S_{\CM}(d) \subseteq S(d)$, we will certainly have
$S'(d) \subsetneq S(d)$ for infinitely many~$d$. The data
in~\cite{hoeij} suggest that this is the case for all $d \ge 9$;
by \cite{DKSS}*{Prop.~1.4}, we know that $S(6) \setminus S'(6) = \{37\}$.

It is perhaps tempting
to assume that for large enough~$d$, the only sporadic\footnote{A degree~$d$
point on a curve~$X$ is \emph{sporadic} if it does not occur in a family
of rational points in the $d$th symmetric power~$X^{(d)}$ of~$X$ parameterized
by a projective space or positive-rank abelian variety.}
points of degree~$d$
on~$X_1(p)$ are CM~points, as this seems to be the expectation for rational
points on modular curves in general. This would imply that $S(d) \subseteq \primes(3d+1)$
for large~$d$. However, the data in the table in~\cite{hoeij} suggest
that there are many sporadic non-CM points (like the degree~$6$ points
on~$X_1(37)$ mentioned above). Still, the bound $p \le 3d+1$
is consistent with this information for $d \ge 13$.

\subsection{New results}

One of our aims in this paper is to show that such a bound for large~$d$
is implied by conjectures on the sparsity of newforms of prime level
and weight~$2$ that have unexpectedly large analytic rank.
More concretely, ``unexpectedly large'' simply means ``positive'' for a newform
with nontrivial character (in this case there is no functional equation forcing
the parity of the rank) and ``at least~$2$'' for a newform with trivial character
and root number~$+1$. See~\cref{conj:str_unif,conj:X0} below.
Under the assumptions of these conjectures, we obtain our first main result.

\begin{theorem}
	Assume the weak forms of~\cref{conj:X0,conj:str_unif}.
	Then $\max S(d) \le 3d + 1$ for $d$ sufficiently large and, in fact, 
	\[ 
	\limsup_{d \to \infty} \frac{\max S(d)}{d} = 3 \,.
	\]
\end{theorem}
The above theorem is a simplified version of the statements in \cref{cor:Sd_bound1,cor:Sd_bound2}, which provide more precise information on $S(d)$.

The techniques used also apply, to a more limited extent, to the study of degree $d$ points on $X_0(p)$.
Specifically, we show (again conditionally) that $X_0(p)$ does not have any non-cuspidal points of small odd degree,
and that all non-cuspidal points of small even degree on $X_0(p)$ must come from points on $X_0(p)^+$. To be more precise, we have the following result.

\begin{theorem}[\cref{thm:X0}]
	The weak form of \cref{conj:X0} implies the following:
	\begin{enumerate}[\upshape(i)]
		\item Fix an odd integer $d \ge 1$. There exists $p_0(d)$ such that for
		all primes $p \ge p_0(d)$, there are no non-cuspidal points of
		degree~$d$ on~$X_0(p)$.
		\item Fix an even integer $d \ge 1$. There exists $p_0(d)$ such that for
		all primes $p \ge p_0(d)$, every non-cuspidal point of
		degree~$d$ on~$X_0(p)$ maps to a point of degree~$d/2$ on~$X_0(p)^+$.
	\end{enumerate}
\end{theorem}

For the necessary background on modular curves, in particular the
definition of the modular curve~$X_H$ between $X_1(p)$ and~$X_0(p)$
for subgroups~$H$ of $\Aut(X_1(p)/X_0(p)) = (\Z/p\Z)^\times/\{\pm 1\}$,
see~\cite{diamondim} or~\cite{DKSS}*{Section~2}.

We also provide numerical evidence for our conjecture on the sparsity of modular forms with unexpectedly high rank. 
\begin{proposition}[\cref{prop:strange_primes}]
	Let $p < 200\,000$ be a prime and
	let $\chi$ be a nontrivial even Dirichlet character modulo~$p$ of order~$n$.
	Then there exists a newform $f \in S_2(\Gamma_1(p), \chi)$ with $L(f,1) = 0$
	if and only if $(p, n)$ is one of the $96$~pairs in~\cref{table:strange_primes}.
\end{proposition}

Thus, numerically it seems that modular forms in~$S_2(\Gamma_1(p), \chi)$ with positive rank
are quite sparse. However, one does expect the existence of infinitely many primes~$p$
for which such a newform exists. Indeed, in~\cref{prop:strange_infinite}
we prove that if $p$ is a prime of the form $(w^2 + w + 11)^2 - 108$, then
there is an even Dirichlet character~$\chi$ modulo~$p$ and an $f \in S_2(\Gamma_1(p), \chi)$
of positive analytic rank. The proof of this involves a geometric construction utilizing
genus~$2$ curves with an automorphism of order~$3$ and a point of infinite order in their Jacobian,
and relies on the modularity of abelian varieties of $\GL_2$-type.
It would be interesting to see if one can provide a similar geometric explanation
for the other primes in~\cref{table:strange_primes}.

In this first part of a series of papers, we will focus on the asymptotic
behavior of the set~$S(d)$ as $d$ tends to infinity. Further papers will
continue the investigation of primes as in~\cref{prop:strange_primes}
and consider specific small values of~$d$.


\subsection*{Acknowledgments}

We would like to thank Drew Sutherland for fruitful discussions and in particular
for his very valuable help with the computational aspects of this study and
Loïc Merel for some information around the gonality lower bound.
The first author was supported by the Croatian Science Foundation under
project No.~IP-2022-10-5008 during this research. We thank the referees
for their useful comments and suggestions.

We used the Magma Computer Algebra System~\cite{Magma} for the computations
mentioned in this paper. Code that verifies the computational claims made
is available on GitHub~\cite{code}.


\subsection*{AI Declaration}

The version of this paper that was originally submitted was written
and its results obtained without any use of AI~tools whatsoever.
While revising the paper, we used AI~tools (mostly Claude Fable~5.1)
to help keep track of the changes, and also to help with finding typos, grammatical
errors and unclear wordings, and potential mathematical gaps.
The revised version was written in its entirety by the authors themselves
(sometimes including suggestions Claude had made).

Claude Code also was used to generate a more (time and memory) efficient
program for detecting strange primes (based on two independent earlier
versions by the authors) and to write additional Magma programs that
verify the computations in Sections 6 and~8. The computational claims
were re-checked using this code. All the code was carefully
reviewed by the authors of this paper, who take full responsibility
for it and for the contents of the paper.


\section{Kernels of Hecke correspondences} \label{sec:Hecke_ker}

Let $N \ge 1$ be an integer and fix a subgroup $H \subseteq G \colonequals (\Z/N\Z)^\times/\{\pm 1\}$.
We will denote the modular curve $X_1(N)/H$ by~$X_H$.
We denote by~$C_H \subseteq X_H$ the subscheme consisting of the (finitely many) cusps.

We recall the following slight strengthening of~\cite{DKSS}*{Prop.~2.3},
using the observation that differences of rational cusps on~$X_H$ are killed by
$T_q - \diamondop{q} - q$ when $q \nmid N$.
Here $T_q$ denotes the $q$th Hecke operator and $\diamondop{q}$ is the diamond
operator associated to~$q$.

\begin{proposition} \label[proposition]{prop:ann_rat_tors}
  Let $q \nmid N$ be a prime and $P \in J_H(\Q)_{\tors}$ such that
  $q$ is odd or $P$ is a sum of a point of odd order and a point in the
  subgroup generated by differences of rational cusps.
  Then $(T_q - \diamondop{q} - q)(P) = 0$.
\end{proposition}

The following is a more general version of~\cite{DKSS}*{Prop.~2.4}.

\begin{proposition} \label[proposition]{prop:hecke_op_kernel}
  Let $N \ge 1$ be an integer
  and fix a subgroup $H \subseteq G \colonequals (\Z/N\Z)^\times/\{\pm 1\}$.
  Let $h_1, h_2 \in \Z_{\ge 0}[G/H][x]$ be polynomials whose coefficients
  are linear combinations of diamond operators on~$X_H$ with nonnegative
  integer coefficients. We assume that $h_1$ is monic and that
  $\deg(h_1) > \deg(h_2) \ge 0$. Let $q \nmid N$ be a prime.
  Then $t_1 = h_1(T_q)$ and $t_2 = h_2(T_q)$ can be considered as (effective)
  correspondences on~$X_H$, and so $t = t_1 - t_2$ induces an endomorphism
  of the divisor group of~$X_H$ over~$\C$. If $D$ is a divisor on~$X_H$
  such that $t(D)$ is supported in cusps (e.g., $t(D) = 0$),
  then $D$ is supported in cusps.
\end{proposition}

\begin{proof}
  We set $d_1 = \deg(h_1)$ and $d_2 = \deg(h_2)$.

  A non-cuspidal point $x \in X_H(\C)$ corresponds to an elliptic
  curve~$E$ over~$\C$ with additional structure.
  The point $\diamondop{a}(x)$ corresponds to
  the same curve~$E$ (with modified extra structure), and $T_q(x)$ is a
  sum of points corresponding to all the elliptic curves that are
  $q$-isogenous to~$E$. We define the \emph{$q$-isogeny graph}~$G_q$
  to have as vertices the isomorphism classes of all elliptic curves
  over~$\C$; two vertices are connected by an edge when there is a
  $q$-isogeny between the corresponding curves. The \emph{distance} between
  two vertices (in the same connected component) is the number of edges
  in a shortest path connecting them.
  There is a natural map $\gamma$ from $X_H(\C) \setminus C_H(\C)$
  to the vertex set of~$G_q$. Let $x$ be a non-cuspidal point in the support
  of a divisor~$D$ on~$X_H$ and let $G_{q,x}$ be the connected component of~$G_q$
  containing~$\gamma(x)$. Let $E$ be the elliptic curve given by~$x$.
  We distinguish two cases.

  First, assume that $E$ does not have~CM. Then $G_{q,x}$ is an infinite
  $(q+1)$-regular tree. The image under~$\gamma$ of the support
  of~$h_1(T_q)(x)$ is contained in the $d_1$-ball around~$\gamma(x)$
  and contains all the vertices at distance~$d_1$ from~$\gamma(x)$
  (this is because $h_1$ is monic),
  whereas the image of the support of~$h_2(T_q)(x)$ is contained
  in the $d_2$-ball and (since $d_2 < d_1$) does not contain vertices
  at distance~$d_1$ from~$\gamma(x)$. Now
  consider a vertex~$v$ of~$G_{q,x}$ that has maximal
  possible distance from~$\gamma(x)$ among all vertices of the
  form~$\gamma(y)$ for a non-cuspidal point~$y$ in the support of~$D$.
  Let $y_1, \ldots, y_n$ be the points in the support of~$D$ such that
  $\gamma(y_j) = v$, and let $w \in G_{q,x}$ be a point at distance~$d_1$
  from~$v$ whose distance from~$\gamma(x)$ is larger by~$d_1$
  than that of~$v$. Each $h_1(T_q)(y_j)$ contains
  precisely one point~$y'_j$ such that $\gamma(y'_j) = w$, and these
  points are distinct for distinct points~$y_j$ (since the map from
  $X_H$ to the $j$-line is étale above the points in~$G_{q,x}$;
  here we use that $E$ does not have~CM).
  Then $w$ is not of the form~$\gamma(z)$ for a point~$z$ in the support
  of~$h_2(T_q)(D)$. This shows that $t(D)$ has non-cuspidal points in its support
  when $D$ has non-CM non-cuspidal points in its support.

  Now consider the case that $E$ has~CM by an order~$\calO$ in an
  imaginary quadratic field. If $q$ is inert in~$\calO$,
  then $G_{q,x}$ is a $(q+1)$-regular tree again, and we can argue
  as before. Otherwise, $G_{q,x}$ has the structure of a ``volcano'';
  see~\cite{sutherland-ants}. For a CM~elliptic curve over~$\C$,
  this volcano has infinite depth. Concretely, this means that it
  consists of a number of rooted $(q+1)$-regular trees whose roots
  form a cycle (of length~$\ge 1$). We can now argue
  as in the first case by choosing~$v$ to be a vertex of maximal
  level (i.e., distance from the root cycle) and~$w$ to be at
  distance~$d_1$ from~$v$ and level larger by~$d_1$. This shows that
  there can be no CM~points in the support of~$D$ as well.
  If $j(x) = 0$ (so $\disc(\calO) = -3$) or $j(x) = 1728$
  (so $\disc(\calO) = -4$), the structure of~$G_{q,x}$ is slightly
  different. It should be considered as a directed graph (with
  edges directed according to the direction of the isogeny).
  The difference then is that the edges pointing away from the
  root cycle have multiplicity~$3$ (resp.,~$2$), whereas all other
  edges are simple. The map from~$X_H$ to the $j$-line is étale
  away from the root cycle, so in particular at the vertex~$w$,
  and it is still true that $t(D)$ has positive coefficients
  at some points mapping to~$w$.

  The only points that we have not excluded from the support of~$D$
  are the cusps; this proves the claim.
\end{proof}


\section{Gonality of modular curves} \label{sec:gon}

One of our workhorse results, \cref{prop:hecke_gon}, is based on a lower
bound for the gonality of the modular curve we want to apply it to.
Recall the following definition.

\begin{definition}
  Let $X$ be a smooth projective and geometrically irreducible curve
  over a field~$k$. The \emph{$k$-gonality of~$X$}, $\gon_k(X)$, is
  the smallest degree of a non-constant rational function on~$X$
  defined over~$k$.
\end{definition}

Clearly, if $k \subseteq K$ is a field extension, then
$\gon_k(X) \ge \gon_K(X)$.

We write $X^{(d)}$ for the $d$th symmetric power of a curve~$X$.
Its points classify effective divisors of degree~$d$ on~$X$; in particular,
the points in~$X^{(d)}(K)$ correspond to $K$-rational effective divisors
of degree~$d$ on~$X$. We will identify effective divisors and points
on~$X^{(d)}$ in this paper without further mention. We write~$[D]$
for the linear equivalence class of a divisor~$D$.

The following is a trivial consequence of the definition above.

\begin{lemma} \label[lemma]{lem:gon_0}
  Let $X$ be a smooth projective and geometrically irreducible curve
  over a field~$k$ and let $x, y \in X^{(d)}(k)$ be two points such that
  $[x-y] = 0$ in the Jacobian of~$X$. If $d < \gon_k(X)$, then $x = y$.
\end{lemma}

\begin{proof}
  The assumption $[x-y] = 0$ says that $x - y$, considered as a
  divisor of degree zero, is the divisor of some rational function
  $f \in k(X)^\times$. If $f$ were non-constant, it would follow that
  $\deg(f) \le d < \gon_k(X)$, a contradiction. So $f$ must be constant,
  which implies that $x - y = 0$.
\end{proof}

We need some information on the gonality of the curves~$X_H$.

\begin{theorem}[Yau, Abramovich, Kim-Sarnak] \label[theorem]{thm:gon_XH}
  {\sloppy
  Let $p$ be a prime number and let \hbox{$H \subseteq (\Z/p\Z)^\times/\{\pm 1\}$}
  be a subgroup. Then
  \[ \gon_{\Q}(X_H) \ge \gamma \frac{p^2 - 1}{2 \#H}
     \qquad\text{with}\qquad
     \gamma = \frac{325}{2^{15}} .
  \]
  In particular,
  \[ \gon_{\Q}(X_1(p)) \ge \gamma \frac{p^2 - 1}{2}
     \qquad\text{and}\qquad
     \gon_{\Q}(X_0(p)) \ge \gamma (p+1) .
  \]}
\end{theorem}

\begin{proof}
  By~\cite{abramovich} and the fact that $(p^2-1)/(2\#H)$ is the degree of
  the map~$\pi_H$ from~$X_H$ to the $j$-line, we have that
  \[ \gon_{\Q}(X_H) \ge \gon_{\C}(X_H) \ge \frac{\lambda_1}{24} \deg(\pi_H)
      = \frac{\lambda_1}{24} \frac{p^2 - 1}{2 \#H} ,
  \]
  where $\lambda_1$ is the smallest positive eigenvalue of the Laplace
  operator on~$X_H(\C)$, and by~\cite{kim-sarnak}, $\lambda_1 \ge 975/4096$.
\end{proof}

The argument given in Abramovich's paper is originally due to Yau (unpublished).

\begin{remark}
  Selberg's Conjecture says that $\lambda_1 \ge \tfrac{1}{4}$. If it holds, then
  we can take $\gamma = \tfrac{1}{96}$ in the bound above.
\end{remark}

Derickx and van Hoeij~\cite{derickx_hoeij} have determined the $\Q$-gonality
of~$X_1(N)$ for $N \le 40$. In these cases, the $\Q$-gonality is achieved
by a modular unit, i.e., a function on~$X_1(N)$ whose zeros and poles are
cusps. We propose the following conjecture (compare Question~1
in~\cite{derickx_hoeij}).

\begin{conjecture} \label[conjecture]{conj:gon}
  Let $p$ be a prime. There is a modular unit~$f$ defined over~$\Q$
  on~$X_1(p)$ such that $\gon_{\Q}(X_1(p)) = \deg f$.
\end{conjecture}

According to~\cite{derickx_hoeij}*{p.~57}\footnote{Actually there it is mentioned
that $\gon_{\Q}(X_1(p)) \leq \left\lfloor \frac{11 p^2}{840} \right\rceil$,
but these values are equal as can be seen by studying the possible values of~$11 p^2$
modulo~$840$.},
we have the upper bound (writing $\lfloor a \rceil$ for the integer closest to~$a$)
\[ \gon_{\Q}(X_1(p)) \leq \left\lfloor \frac{11 (p^2-1)}{840} \right\rceil \,. \]

In fact, at the time of writing, only three values of~$p$ are known for which $X_1(p)$ has a
$\Q$-rational function of degree $< \left\lfloor \frac{11 (p^2-1)}{840} \right\rceil$.
These values are $\set{31, 67, 101}$; in these cases a function of degree
$\left\lfloor \frac{11 (p^2-1)}{840} \right\rceil - 1$ is known according to
the table in~\cite{derickx_hoeij}.

In this light, it is interesting to ask the question whether the limit
\[ \lim_{p \to \infty} \frac{\gon_{\Q}(X_1(p))}{p^2 - 1} \]
exists, and if it does, whether it is closer to the upper bound $\frac{11}{840} \approx 0.0131$
or to the lower bound $\frac{325}{2^{16}} \approx 0.00496$ (or $\frac{1}{192} \approx 0.00521$
under Selberg's Conjecture).

The known growth of the gonality of~$X_1(p)$ implies the following.

\begin{proposition} \label[proposition]{prop:Sprime}
  There are constants $C_1' \ge C_1 > 0$ such that for all $d \ge 1$,
  \[ \primes(\sqrt{C_1 d + 1}) \subseteq S'(d) \subseteq \primes(\sqrt{C_1' d + 1}) \,. \]
\end{proposition}

\begin{proof}
  By a result of Frey~\cite{frey}*{Prop.~2}, $p \in S'(d)$ implies $\gon_{\Q}(X_1(p)) \le 2d$.
  Combined with~\cref{thm:gon_XH}, this gives the upper bound (with $C_1' = 4\gamma^{-1}$).
  Conversely, if $X$ is a curve of genus~$g$ over~$\Q$ with a rational point,
  then the Riemann-Roch Theorem implies that there are functions in~$\Q(X)^\times$
  of exact degree~$d$ for each $d \ge \max\{1, 2g\}$ (with a pole of order~$d$
  at the chosen rational point), so there are infinitely many points
  of degree~$d$ on~$X$ (this uses the Hilbert Irreducibility Theorem).
  Since the genus of~$X_1(p)$ is $(p - 5) (p - 7)/24$ for $p \ge 5$
  and~$0$ for $p = 2, 3$, this gives the lower bound
  with $C_1 = 12$: if $d \ge (p^2 - 1)/12$, then $d \ge \max\{1, 2 g(X_1(p))\}$.
\end{proof}


\section{A criterion for ruling out moderately large primes} \label{sec:modlarge}

In~\cite{DKSS}*{Prop.~7.1}, we gave a criterion, in terms of the gonality,
for a point on~$X_1(p)$ of a given degree to arise as a pull-back from
an intermediate modular curve.
We extend this result to intermediate curves~$X_H$ and more general
not necessarily prime level~$N$, which we anticipate will be useful
for future applications (although we will only require the case~$X_1(p)$
with $p$~prime in this paper).

\begin{proposition} \label[proposition]{prop:hecke_gon}
	Let $d \ge 1$, let $N$ be an integer, $\ell \nmid N$ a prime and let $H$ be a subgroup
	of~$(\Z/N\Z)^\times/\{\pm 1\}$. Let $a \in \bigl((\Z/N\Z)^\times/\{\pm 1\}\bigr)/H$
	be such that
	\[ A \colonequals (\diamondop{a} - 1)(J_H(\Q)) \]
	is finite.
	When $\ell = 2$, we additionally assume that $A[2^\infty]$ is killed by $T_2 - \diamondop{2} - 2$
	(by~\cref{prop:ann_rat_tors}, this follows when $A[2^\infty]$ is contained in
	the subgroup generated by differences of rational cusps). We set
	\[ n = \begin{cases}
	          2\ell + 1 & \text{if $a \in \set{\ell, \ell^{-1}}$,} \\
	          2\ell + 2 & \text{if $a \notin \set{\ell, \ell^{-1}}$.}
	       \end{cases}
	\]
	Then any rational point on~$X_H^{(d)}$ of degree $d < \gon_{\Q}(X_H)/n$
	and without cusps in its support is a sum of orbits under~$\diamondop{a}$.
\end{proposition}

\begin{remark}
  By~\cref{thm:gon_XH}, the inequality $d < \gon_{\Q}(X_H)/n$ above holds when
  $N = p$ is prime and
  \begin{equation*}\label{DS-inequality}
    d < \frac{325}{2^{15}} \, \frac{p^2 - 1}{2 n \cdot \#H} \,.
  \end{equation*}

  We note that when $N = p > 3$ is prime, we can always take $\ell = 3$; then $n \le 8$.
\end{remark}

\begin{proof}
  Compare the proof of~\cite{DKSS}*{Prop.~7.1}.
  Our assumptions together with~\cref{prop:ann_rat_tors} imply that
  $T_\ell - \diamondop{\ell} - \ell$ kills~$A$.
  Let $x \in X_H^{(d)}(\Q)$ without cusps in its support. Pick a rational cusp $c \in X_H$.
  Then $(\diamondop{a} - 1) [x] = (\diamondop{a} - 1) [x - d \cdot c] + d (\diamondop{a}[c] - [c])$
  is a sum of an element of~$A$ and a multiple of a difference of two cusps, both of
  which are killed by $T_\ell - \diamondop{\ell} - \ell$, so the same is true for $(\diamondop{a} - 1) [x]$.
  Rearranging the resulting equation, we obtain the following linear equivalence of effective divisors on~$X_H$.
  \[ (\diamondop{a} T_\ell + \diamondop{\ell} + \ell) \cdot x
       \sim (T_\ell + \diamondop{\ell a} + \ell \diamondop{a}) \cdot x
  \]
  If $a \in \{\ell, \ell^{-1}\}$, then we can cancel $\diamondop{\ell} x$
  or $x$, so that the divisors on both sides of the resulting equivalence
  have degree $(2\ell + 1) d$; otherwise
  they have degree $(2\ell + 2) d$. So in both cases, this degree is~$nd$.
  By~\cref{lem:gon_0}, it follows that both sides are equal as divisors,
  so
  \[ (T_\ell - \diamondop{\ell} - \ell) \cdot (\diamondop{a} x - x) = 0 \,. \]
  Then \cref{prop:hecke_op_kernel} implies that $\diamondop{a} x - x$ is supported
  on cusps. As $x$ does not contain cusps in its support by assumption, it follows
  that $\diamondop{a} x = x$, which is equivalent to saying that $x$ is a sum
  of orbits of~$\diamondop{a}$.
\end{proof}

Regarding the extra condition when $\ell = 2$ in the case $X_H = X_1(p)$,
we quote the following, which is Conjecture~6.2.2 in~\cite{CES}.

\begin{conjecture} \label[conjecture]{conj:CES}
  Let $p$ be a prime. Then the rational torsion subgroup
  of~$J_1(p)$ is generated by differences of rational cusps.
\end{conjecture}

\begin{remark} \label[remark]{rmk:CES}
  In~\cite{CES}, this is shown for all primes $p \le 157$ with
  the exception of
  \[ p \in \set{29, 97, 101, 109, 113} , \]
  and in these
  cases, they bound the index of the subgroup generated by differences
  of rational cusps by $2^6$, $17$, $2^4$, $3^7$, and $2^{12} \cdot 3^2$,
  respectively. In~\cite{DKSS}*{Thm.~3.2}, the conjecture is shown
  for $p = 29$. Davide De~Leo establishes the conjecture for the remaining
  open cases (for $p \le 157$) in his Master's thesis~\cites{DeLeo,DLS}.
  This implies in particular that the $2$-primary part of~$J_1(p)(\Q)_{\tors}$
  is contained in the subgroup generated by differences of rational
  cusps for all $p \le 157$.
\end{remark}

We fix $N = p$ to be a prime in the following.

\begin{remark} \label[remark]{rmk:orbits}
  The orbits under~$\diamondop{a}$ on~$X_1(p)$ have
  length~$\ord(a)$ (the order of~$a$ as an element of $(\Z/p\Z)^\times/\{\pm 1\}$) unless
  \begin{enumerate}[$\bullet$]
    \item $3 \mid \ord(a)$, so necessarily $p \equiv 1 \bmod 6$, and the points
          map on~$X_0(p)$ to points corresponding to pairs $(E, C)$ with
          $j(E) = 0$ and $C$ the kernel of an element
          $\pi \in \End_{\C}(E) \simeq \Z[\omega]$ of norm~$p$, where
          $\omega$ is a primitive cube root of unity;
          such an orbit has length $\ord(a)/3$, or
    \item $2 \mid \ord(a)$, so necessarily $p \equiv 1 \bmod 4$, and the points
          map on~$X_0(p)$ to points corresponding to pairs $(E, C)$ with
          $j(E) = 1728$ and $C$ the kernel of an element
          $\pi \in \End_{\C}(E) \simeq \Z[i]$ of norm~$p$, where
          $i$ is a primitive fourth root of unity;
          such an orbit has length $\ord(a)/2$.
  \end{enumerate}
\end{remark}

We consider the case that $x \in X_1(p)^{(d)}(\Q)$ comes from a non-cuspidal
point of degree~$d$ on~$X_1(p)$, say corresponding to the pair~$(E, P)$
of an elliptic curve~$E$ defined over a number field~$K$ of degree~$d$
and a point $P \in E(K)$ of order~$p$, such that $K = \Q(E, P)$
(i.e., the Galois action on the support of~$x$ is transitive).
Then the conclusion of~\cref{prop:hecke_gon} implies that for every
$m \in \Z$ with $p \nmid m$ such that the image of~$m$
in~$(\Z/p\Z)^\times/\{\pm 1\}$ is in the subgroup generated by~$a$,
$(E, mP)$ is isomorphic to a Galois conjugate of~$(E, P)$.
This fact can be expressed as follows.
Recall that we can always take ($\ell = 3$ and) $n = 8$ in~\cref{prop:hecke_gon}.

\begin{corollary} \label[corollary]{cor:medium_gen}
  Let $d \ge 1$ be an integer and let $p$ be a prime
  such that $8 d < \gon_{\Q}(X_1(p))$.
  Assume that there is some nontrivial subgroup $H$
  of~$(\Z/p\Z)^\times/\{\pm 1\}$ such that all simple factors of~$J_1(p)$
  of positive (analytic) rank are factors of~$J_H$.

  Let $x \in X_1(p)^{(d)}(\Q)$ be a point without cusps in its support.
  Then either $x$ contains points with $j$-invariant $0$ or~$1728$,
  or $d = m \cdot \#H$ for some integer~$m$ and $x$ arises as the pull-back
  of a point in~$X_H^{(m)}(\Q)$.
\end{corollary}

\begin{proof}
  We take $a$ to be a generator of (the cyclic group)~$H$; then
  $A$ in~\cref{prop:hecke_gon} is finite since by assumption, all its isogeny
  factors have analytic rank and hence Mordell-Weil rank zero~\cite{kato}.
  The gonality condition in~\cref{prop:hecke_gon} is also satisfied by assumption.
  So $x$ is a sum of $H$-orbits. If the support of~$x$ does not
  contain points with $j$-invariant $0$ or~$1728$,
  then all these $H$-orbits in~$x$ have length~$\#H$, which shows that
  $d = m \cdot \#H$ for some integer~$m$. This also implies that $x$ is
  the pull-back of a point in~$X_H^{(m)}(\Q)$.
\end{proof}

The following lemma gives more precise information on when we can apply this corollary.

\begin{lemma} \label[lemma]{lem:simple_factors_of_JH}
  Let $f_1, \ldots, f_n$ be the newforms in~$S_2(\Gamma_1(p))$ that have positive analytic rank,
  and let $\chi_1, \ldots, \chi_n$ be their nebentypus characters.
  Let $H \leqslant (\Z/p\Z)^\times/\set{\pm 1}$ be a subgroup. Then all simple factors of~$J_1(p)$
  of positive (analytic) rank are factors of~$J_H$ if and only if
  \[ H \leqslant \bigcap_{i=1}^n \ker \chi_i \,. \]
  Note that the characters are even, so we can consider them as characters
  on~$(\Z/p\Z)^\times/\set{\pm 1}$.
\end{lemma}

\begin{proof}
  Define $\Gamma_H \subseteq \Gamma_0(p)$ to be the congruence subgroup given by the matrices
  of the form $\left( \begin{smallmatrix} a & b \\ c & d \end{smallmatrix}\right)$ such that $a \in H$.
  Then the simple isogeny factors of~$J_H$ are exactly the abelian varieties~$A_f$ corresponding
  to the newforms in~$S_2(\Gamma_H)$. Now $(\Z/p\Z)^\times$ acts on~$S_2(\Gamma_1(p))$
  through the diamond operators and one has that $S_2(\Gamma_H) = S_2(\Gamma_1(p))^H$,
  where $S_2(\Gamma_1(p))^H$ is the space of forms invariant under the action of~$H$.
  Now the nebentypus character describes the action of~$H$ on~$S_2(\Gamma_1(p))$;
  indeed, if $a \in (\Z/p\Z)^\times$, then $\diamondop a f_i =\chi_i(a)f_i$.
  In particular, $f_i \in S_2(\Gamma_H)$ if and only if $H \leqslant \ker \chi_i$,
  and the statement follows.
\end{proof}

We would like to take $H = (\Z/p\Z)^\times/\{\pm 1\}$. We can do that
unless there is a positive-rank factor of~$J_1(p)$ that corresponds to
an orbit of newforms with nontrivial character. Such an orbit exists
by definition if and only if $p$ is strange in the sense
of~\cref{sec:strange} below. We will come back to the implications
for~$S(d)$ in~\cref{Sect:appl_strange}.


\section{Strange primes} \label{sec:strange}

In view of the criterion discussed in~\cref{sec:modlarge}, we
make the following definition.
We state it for general levels~$N$, even though we will be only concerned
with prime levels in this paper.

\begin{definition} \label[definition]{def:strange}
  Let $N \ge 1$ be an integer, $\chi$ a Dirichlet character of modulus~$N$
  and let $f$ be a newform for~$\Gamma_1(N)$ of weight~$2$ and character~$\chi$.
  \begin{enumerate}[(1)]
  	\item The newform~$f$ is \emph{strange} if $\chi$ is
          nontrivial and $L(f, 1) = 0$, i.e., the analytic rank of the associated
          abelian variety~$A_f$ is positive.
    \item The character~$\chi$ is \emph{strange} if there exists a divisor $M \mid N$
          and a strange newform~$f$ of weight~$2$ on~$\Gamma_1(M)$ with character~$\chi'$
          such that $\chi$ is the character modulo~$N$ induced by~$\chi'$.
    \item $N$ is \emph{strange} if there is a strange newform whose level divides~$N$.
    \item The \emph{strangeness} $\str(N)$ of~$N$ is the order of the group generated
          by all strange characters mod~$N$. In particular, $\str(N) > 1$ if and only if $N$ is strange.
    \item The \emph{new strangeness dimension} $\sdim_{\text{new}}(N)$ of~$N$ is
          the number of strange newforms at level~$N$.
    \item The \emph{strangeness dimension} $\sdim(N)$ is defined as
          \[ \sdim(N) = \sum_{M\mid N} \tau(N/M) \sdim_{\text{new}}(M) \,, \]
          where $\tau(n)$ denotes the number of divisors of~$n$.
          In particular, $\sdim(N) > 0$ if and only if $N$ is strange.
  \end{enumerate}
\end{definition}

The formula for the definition of $\sdim$ in terms of $\sdim_{\text{new}}$ is motivated
by the Atkin-Lehner-Li decomposition
\[ S_k(\Gamma_1(N)) \cong \bigoplus_{M \mid N} \bigoplus_{d \mid N/M} S_k(\Gamma_1(M))_{\text{new}} \,. \]
It counts a strange newform of level $M \mid N$ with the multiplicity it occurs in~$S_2(\Gamma_1(N))$.

Since there are no newforms of level $1$ and weight $2$ it follows that for a prime~$p$ one has
$\sdim(p) = \sdim_{\text{new}}(p)$.

Whether a newform is strange depends only on its Galois orbit;
see~\cref{prop:galois_invariance_of_strangeness} below. We first state a lemma that will allow us to translate the analytic statement on
the order of vanishing of~$L(f, s)$ at~$s = 1$ into an algebraic statement.

Fix $N \ge 1$. Let $\T$ denote the \emph{Hecke algebra}, i.e., the subring of the
endomorphism ring of~$S_2(\Gamma_1(N))$ generated by the Hecke operators~$T_p$ for
primes $p \nmid N$. For any ring~$R$, we write $\T_R = \T \otimes_{\Z} R$.
We fix an embedding $\Qbar \hookrightarrow \C$.
Let $f \in S_2(\Gamma_1(N))$ be a (normalized) newform.
Then it is a well-known result (``multiplicity one'') that there is an idempotent
$e_f \in \T_{\Qbar} \subseteq \T_{\C}$ such that $e_f \cdot S_2(\Gamma_1(N)) = \C f$.
This means that for $p \nmid N$, $T_p e_f = a_p(f) e_f$.
The absolute Galois group of~$\Q$ permutes the idempotents~$e_f$ in the same way
as it permutes the newforms~$f$.

The lemma below is a variation on the vanishing condition in~\cite{parent1}*{Lemma~3.10}
that is more suitable for our applications. For a ring~$R$, we denote
by~$\MMM_2(\Gamma_1(N), R)^+$ the plus-part of the space of $R$-valued
modular symbols on~$\Gamma_1(N)$, and by~$\SSS_2(\Gamma_1(N), R)^+$
its cuspidal subspace. Integration gives a Hecke compatible pairing
\[ \langle \cdot, \cdot \rangle \colon S_2(\Gamma_1(N)) \times \MMM_2(\Gamma_1(N), \C)^+ \to \C \]
whose restriction to $S_2(\Gamma_1(N)) \times \SSS_2(\Gamma_1(N), \C)^+$
is non-degenerate; see~\cite{stein-modform}*{Theorem~8.23}.
We denote by~$\bfe$ the image of the winding element
$-\set{0, \infty} \in\MMM_2(\Gamma_1(N), \Z)^+$ under the map
\[ \MMM_2(\Gamma_1(N), \C)^+ \to
      \operatorname{Hom}_\C\bigl(S_2(\Gamma_1(N)), \C\bigr) \cong
      \SSS_2(\Gamma_1(N), \C)^+
\]
induced by the integration pairing; $\bfe$ is rational by the
Manin--Drinfeld Theorem; see \cite{DKSS}*{Definition~2.1}.
Note that $\T$ acts in a natural way on~$\SSS_2(\Gamma_1(N), \Q)^+$.

\begin{lemma} \label[lemma]{lem:vanishing_criterium}
  Let $f$ be as above. Then $L(f, 1) = 0$ if and only if $e_f \cdot \bfe = 0$.
\end{lemma}

\begin{proof}
  Under the integration pairing,
  \[ L(f, 1) = -2 \pi i \langle f, \set{0, \infty} \rangle = 2 \pi i \langle f, \bfe \rangle \,. \]
  So $L(f, 1) = 0$ iff $\langle f, \bfe \rangle = 0$. Since $e_f \cdot S_2(\Gamma_1(N)) = \C f$,
  one has that $\langle f, \bfe \rangle = 0$ if and only if for all $g\in  S_2(\Gamma_1(N))$ one has
  $\langle e_f \cdot g, \bfe \rangle = 0$. Hecke compatibility implies
  $\langle e_f \cdot g, \bfe \rangle=\langle  g, e_f \cdot\bfe \rangle $,
  so that the lemma follows from the pairing being non-degenerate
  on $S_2(\Gamma_1(N)) \times \SSS_2(\Gamma_1(N), \C)^+$.
\end{proof}

\begin{proposition} \label[proposition]{prop:galois_invariance_of_strangeness}
  Let $N$ be an integer and $f_1, f_2 \in S_2(\Gamma_1(N))$ a pair of Galois conjugate newforms.
  Then $L(f_1, 1) = 0$ if and only if $L(f_2, 1) = 0$. Similarly, if $f_1, f_2 \in S_2(\Gamma_0(N))$,
  then $f_1$ has positive even analytic rank if and only if $f_2$ has positive even analytic rank.
\end{proposition}

\begin{proof}
  To lighten notation, we write $e_i$ for the idempotent~$e_{f_i}$.
  Since $f_1$ and~$f_2$ are Galois conjugate, the idempotents $e_1$ and~$e_2$
  are as well (under the natural Galois action on $\T_{\Qbar}$).
  The winding element~$\bfe$ is fixed by the Galois action, hence $e_1 \cdot \bfe$
  and $e_2 \cdot \bfe$ are also Galois conjugate. Using~\cref{lem:vanishing_criterium}
  we obtain the following chain of equivalences.
  \[ L(f_1, 1) = 0 \iff e_1 \cdot \bfe = 0 \iff e_2 \cdot \bfe = 0 \iff L(f_2, 1) = 0 \,. \]
  This proves the first claim. For the second claim, we also use that the eigenvalue
  under the Fricke involution (and hence the parity of the analytic rank) is constant
  on Galois orbits.
\end{proof}

\begin{remark}
  The sum $e'_f = \Tr e_f \in \T_{\Q}$ of the idempotents associated to the newforms in the Galois orbit
  of~$f$ is an idempotent defined over~$\Q$ such that $\T_{\Q} e'_f \cong \Q(f)$,
  the field of coefficients (= Hecke eigenvalues) of~$f$. We see from the proof above that
  $e_f \cdot \bfe = 0 \iff e'_f \cdot \bfe = 0$. This gives a ``rational'' condition
  that allows us to check computationally whether $f$ is strange or not; see below.
\end{remark}

\begin{corollary} \label[corollary]{cor:charpoly_integral}
  Let $N \ge 1$ be an integer. Denote by $\mathscr{S}(N)$ the set of all strange newforms
  for~$\Gamma_1(N)$. Let further $\ell$ be some prime. Then the polynomial
  \[ h_N(x) = \prod_{f \in \mathscr{S}(N)} (x - a_\ell(f)) \]
  has coefficients in~$\Z$. The corresponding statement is valid for the set of newforms
  of positive even analytic rank for~$\Gamma_0(N)$.
\end{corollary}

\begin{proof}
  This follows from the fact that the Galois action on the newforms induces
  a Galois action on the coefficients (which are algebraic integers)
  and from~\cref{prop:galois_invariance_of_strangeness},
  which shows that $\mathscr{S}(N)$ is a union of Galois orbits.
\end{proof}

The \href{http://www.lmfdb.org/ModularForm/GL2/Q/holomorphic/?level_type=prime&weight=2&char_order=2-&analytic_rank=1-&search_type=List}{LMFDB}
has complete data of all newforms of weight~$2$ up to level~$1000$.
In this range, there are only nine strange primes~$p$, and for each
of these primes, there is exactly one Galois orbit of strange newforms. They are given
in~\cref{table:newforms}. The column labeled ``$\ord(\chi)$'' gives the order
of the associated character. Since there is only one orbit of strange newforms per level,
\cref{lem:simple_factors_of_JH} implies that $\ord(\chi)$ is equal to the index of the largest
subgroup~$H$ we can take in~\cref{cor:medium_gen} for the prime~$p$.
We also give the order of this maximal~$H$.

\begin{table}[htb]
\[ \renewcommand{\arraystretch}{1.2}
  \begin{array}{|c|c|c|c|} \hline
    \text{level} & \ord(\chi) & \dim A_f & \#H \\\hline
      \text{\href{http://www.lmfdb.org/ModularForm/GL2/Q/holomorphic/61/2/f/a/}{61}}   &  6 & 2 &  5 \\
      \text{\href{http://www.lmfdb.org/ModularForm/GL2/Q/holomorphic/97/2/g/a/}{97}}   & 12 & 4 &  4 \\
      \text{\href{http://www.lmfdb.org/ModularForm/GL2/Q/holomorphic/101/2/e/a/}{101}} & 10 & 4 &  5 \\
      \text{\href{http://www.lmfdb.org/ModularForm/GL2/Q/holomorphic/181/2/f/a/}{181}} &  6 & 2 & 15 \\
      \text{\href{http://www.lmfdb.org/ModularForm/GL2/Q/holomorphic/193/2/g/a/}{193}} & 12 & 4 &  8 \\
      \text{\href{http://www.lmfdb.org/ModularForm/GL2/Q/holomorphic/409/2/g/a/}{409}} & 12 & 4 & 17 \\
      \text{\href{http://www.lmfdb.org/ModularForm/GL2/Q/holomorphic/421/2/f/a/}{421}} &  6 & 2 & 35 \\
      \text{\href{http://www.lmfdb.org/ModularForm/GL2/Q/holomorphic/733/2/e/a/}{733}} &  6 & 2 & 61 \\
      \text{\href{http://www.lmfdb.org/ModularForm/GL2/Q/holomorphic/853/2/e/a/}{853}} &  6 & 2 & 71 \\\hline
  \end{array}
\]
\caption{Newforms~$f$ of weight~$2$, nontrivial character~$\chi$ and prime level
         $p < 1000$ such that $L(f, 1) = 0$.} \label{table:newforms}
\end{table}

We have extended the range of the LMFDB data by a computation, as follows.
Let $p$ be the prime we want to test for strangeness. Pick a (reasonably
large, but not too large, say between~$2^{22}$ and~$2^{23.5}$; the upper
bound is the threshold below which Magma's finite-field arithmetic is fast)
prime $q \equiv 1 \bmod p-1$.
Working mod~$q$, all the Dirichlet characters mod~$p$ take values
in~$\F_q^\times$. So we can compute the space of modular symbols
over~$\F_q$ associated to any given (even) Dirichlet character~$\chi$
mod~$p$ (we take the subspace fixed by the star involution). We know that
$(T_\ell - \ell \diamondop{\ell} - 1)(T_\ell - \diamondop{\ell} - \ell)$
maps the modular symbol $\bfe = -\set{0, \infty}$ into the cuspidal subspace,
for every prime~$\ell$. We find the smallest~$\ell \neq p$ such that the resulting
element is nonzero (almost always $\ell = 2$). Write $\bfe'$ for this
element of the cuspidal subspace. We compute the subspace $\F_q[T_2] \cdot \bfe'$
(as the Krylov subspace of~$\bfe'$ under~$T_2$) and enlarge it by saturating it
under the action of the~$T_\ell$ with $\ell < 30$. If the resulting subspace
is the full cuspidal subspace, then
the projection of the winding element~$\bfe$ into any of the newform
spaces associated to~$\chi$ (over~$\Q$) is nonzero, and it follows
that $\chi$ is not strange. If, on the other hand, we obtain a smaller
subspace, then it is quite likely that $\chi$ is indeed strange
(since $q$ is taken to be reasonably large); the difference~$d$ in dimension
between the full cuspidal space and the subspace gives a bound on the
number of strange newforms with character~$\chi$. To rigorously prove that
there are exactly~$d$ such newforms,
we perform a similar (but much slower) computation in characteristic zero
over~$\Q(\zeta_n)$, where $n$ is the order of~$\chi$.
This works because $(T_\ell - \ell \diamondop{\ell} - 1)(T_\ell - \diamondop{\ell} - \ell)$
is invertible on the space associated to any newform~$f$:
we have $a_\ell(f) \notin \{\ell \chi(\ell) + 1, \chi(\ell) + \ell\}$;
otherwise $a_\ell(f) = \chi(\ell) \overline{a_\ell(f)}$
(see~\cite{ribet1976}*{page~21}) would imply $\chi(\ell) = 1$ and
$a_\ell(f) = \ell + 1$, but $|a_\ell(f)| \le 2 \sqrt{\ell}$.
(For details, see \cite{code}*{file \texttt{strange\_primes.magma}}.)

Drew Sutherland, using code written by the first author of this paper,
found all candidates for strange characters modulo primes $p < 10^5$.
The second author used independently written code to corroborate these
results for $p < 50\,000$ and to verify that candidates for strange
characters are indeed strange. These computations were done using the
Modular Symbols functionality of Magma~\cite{Magma}.
We have re-checked this computation by running the new version of the program
produced by Claude Code under our guidance; the computation took roughly
half a week on a machine with 16~cores. We then extended the computation
to all primes below~$200\,000$ within two further weeks on the same machine.
The result is the following.

\begin{proposition} \label[proposition]{prop:strange_primes}
  Let $p < 200\,000$ be a prime and
  let $\chi$ be an even Dirichlet character modulo~$p$.
  The character~$\chi$ is strange if and only if $p$ and
  the order of~$\chi$ are listed in~\cref{table:strange_primes}.
\end{proposition}

\begin{table}[htb]
\[ \renewcommand{\arraystretch}{1.15}
   \setlength{\doublerulesep}{3pt} 
   \setlength{\arraycolsep}{4pt} 
  \begin{array}{|r|*{9}{c|}} \hline
    p          &     61 &     97 &    101 &    181 &    193 &    409 &    421 &    733 &    853 \\\hline
    \ord(\chi) &      6 &     12 &     10 &      6 &     12 &     12 &      6 &      6 &      6 \\\hline
    \sdim(p)   &      2 &      4 &      4 &      2 &      4 &      4 &      2 &      2 &      2 \\\hline
    \hline
    p          &   1021 &   1777 &   1801 &   1861 &   2377 &   2917 &   3229 &   3793 &   4201 \\\hline
    \ord(\chi) &     30 &      3 &      5 &      6 &     12 &      6 &      3 &     12 &      3 \\\hline
    \sdim(p)   &      8 &      2 &      4 &      6 &      4 &      2 &      2 &      4 &      2 \\\hline
    \hline
    p          &   4733 &   5441 &   5821 &   5953 &   6133 &   6781 &   7477 &   8681 &   8713 \\\hline
    \ord(\chi) &      7 &     10 &      6 &      3 &      6 &      6 &     14 &     10 &  4, 12 \\\hline
    \sdim(p)   &      6 &      4 &      2 &      2 &      2 &      2 &      6 &      4 &  4 + 4 \\\hline
    \hline
    p          &  10093 &  11497 &  12941 &  14533 &  15061 &  15289 &  17041 &  17053 &  17257 \\\hline
    \ord(\chi) &      6 &      3 &     10 &      6 &      6 &     12 &      3 &      6 &     12 \\\hline
    \sdim(p)   &      2 &      2 &      4 &      2 &      4 &      4 &      2 &      2 &      4 \\\hline
    \hline
    p          &  18199 &  20341 &  22093 &  23017 &  23593 &  26161 &  26177 &  28201 &  29569 \\\hline
    \ord(\chi) &      3 &      6 &      6 &     12 &     12 &      3 &      4 &      3 &      2 \\\hline
    \sdim(p)   &      4 &      2 &      2 &      4 &      4 &      2 &      4 &      2 &      2 \\\hline
    \hline
    p          &  31033 &  31657 &  32497 &  35521 &  35537 &  36373 &  39313 &  41081 &  41131 \\\hline
    \ord(\chi) &      3 &      3 &      3 &      3 &      4 &      6 &     12 &      5 &      3 \\\hline
    \sdim(p)   &      2 &      2 &      2 &      2 &      4 &      2 &      4 &      4 &      2 \\\hline
    \hline
    p          &  41593 &  42793 &  48733 &  52561 &  52691 &  53113 &  53857 &  63313 &  63901 \\\hline
    \ord(\chi) &     12 &      3 &      6 &      3 &      5 &     12 &     12 &     12 &      6 \\\hline
    \sdim(p)   &      4 &      2 &      2 &      2 &      4 &      4 &      4 &      4 &      2 \\\hline
    \hline
    p          &  65171 &  65449 &  66973 &  68737 &  69061 &  69401 &  69457 &  73009 &  86113 \\\hline
    \ord(\chi) &      5 &     12 &      6 &     12 &      6 &      5 &      4 &     12 &     12 \\\hline
    \sdim(p)   &      4 &      4 &      2 &      4 &      2 &      4 &      4 &      4 &      4 \\\hline
    \hline
    p          &  86161 &  96289 & 105517 & 105751 & 109453 & 114553 & 114889 & 131203 & 133117 \\\hline
    \ord(\chi) &      4 &     12 &      6 &      5 &      6 &      3 &     12 &      3 &      6 \\\hline
    \sdim(p)   &      4 &      4 &      2 &      4 &      2 &      2 &      4 &      4 &      2 \\\hline
    \hline
    p          & 134257 & 139021 & 140401 & 151213 & 151429 & 155593 & 159073 & 161977 & 162529 \\\hline
    \ord(\chi) &     12 &      6 &      3 &      6 &      6 &     12 &      3 &     12 &      3 \\\hline
    \sdim(p)   &      4 &      2 &      2 &      2 &      2 &      4 &      2 &      4 &      2 \\\hline
    \hline
    p          & 173431 & 178897 & 192613 & 195817 & 199873 &        &        &        &        \\\hline
    \ord(\chi) &      5 &      3 &      6 &     12 &      3 &        &        &        &        \\\hline
    \sdim(p)   &      4 &      2 &      2 &      4 &      2 &        &        &        &        \\\hline
  \end{array}
\]
\caption{Strange primes $p < 200\,000$, orders of strange characters mod~$p$,
         and the strangeness dimension of~$p$. Each dimension given corresponds
         to one Galois orbit of newforms.}
\label{table:strange_primes}
\end{table}
There are just~$95$ strange primes up to~$2 \cdot 10^5$.
Why should we expect strange primes to be rare? The $L$-series of
a newform~$f$ for~$\Gamma_0(p)$ satisfies a functional equation
$L(f, 2-s) = \pm L(f, s)$ with sign the negative of the eigenvalue
of~$f$ under the Fricke involution~$w_p$.
So the analytic rank must be odd when $f$ is invariant under~$w_p$.
When the character~$\chi$ of~$f$ is nontrivial, however, the coefficient field
of~$f$ is no longer totally real, but instead totally complex
(this follows from the relation $a_\ell = \chi(\ell) \bar{a}_\ell$ for primes
$\ell \neq p$; see~\cite{ribet1976}*{page~21}), and
the functional equation has the form $L(\bar{f}, 2-s) = \eps L(f, s)$
with $|\eps| = 1$. So the functional equation does not force a zero
at $s = 1$, and we can expect it to be rare for a zero to occur,
and particularly so when the Galois orbit of~$f$ is large.

One can also use similar heuristics as in Section~\ref{sec:X0} below:
A strong version of the analogue of Maeda's Conjecture for the newforms
of level~$p$ with a given non-trivial character would say that there should
only be few small Galois orbits and one very large one (which is consistent
with experimental observations), and there may be
reason to believe that the total size of the small orbits grows only
slowly or is even uniformly bounded. An analogue of the result of
Iwaniec and Sarnak mentioned in Section~\ref{sec:X0} would then imply
that $\sdim(p)$ is bounded by the total size of the small orbits.

Based on the data and heuristics above, we propose the following conjecture.

\begin{conjecture} \label[conjecture]{conj:str_unif} \strut
  \begin{enumerate}[\upshape(1)]
    \item \label{conj:str_unif_3}
          (Weak form)
          \[ \lim_{p \to \infty} \frac{\sdim(p)}{\log p} = 0 \]
          as $p$ runs through all prime numbers.
    \item \label{conj:str_unif_2}
          (Strong form)
          The strangeness dimension~$\sdim(p)$ is uniformly bounded
          as $p$ runs through all prime numbers.
  \end{enumerate}
\end{conjecture}

Part~\eqref{conj:str_unif_2} of this conjecture implies that $\str(p)$
is uniformly bounded for all primes~$p$, since the size of the Galois
orbit of a strange newform is a multiple of~$\varphi(\ord(\chi))$,
where $\chi$ is the associated strange character and $\varphi$ is
the Euler totient function. So a bound on~$\sdim(p)$ implies a
bound on~$\ord(\chi)$ and therefore also a bound on~$\str(p)$.

The strangeness dimensions of all primes below~$2 \cdot 10^5$ are bounded
by~$8$ ($8$ occurs twice and $6$ three times, all for $p < 10\,000$;
for $p > 10\,000$ all strangeness dimensions are at most~$4$),
and the characters have order bounded by~$30$
(occurring once; $14$ occurs once, all other orders are at most~$12$).
From the data, it therefore appears to be possible that
\[ \max_p \str(p) = 30 \qquad\text{and}\qquad \limsup_{p \to \infty} \str(p) = 12 \]
and that
\[ \max_p \sdim(p) = 8 \qquad\text{and}\qquad \limsup_{p \to \infty} \sdim(p) = 4 . \]


\section{A source of strange primes} \label{sec:strange_examples}

In the following, we describe a source of strange primes~$p$ such that
the character of the strange newforms has order~$6$ and the Galois orbit
of these newforms has length~$2$. This case occurs several times
in~\cref{table:strange_primes}, including for fairly large primes.

Consider a curve~$X$ of genus~$2$ over~$\Q$ such that $X$ has an
automorphism~$\sigma$
of order~$3$ defined over~$\Q$. Denote the Jacobian variety of~$X$ by~$J$.
Since the hyperelliptic involution~$\iota$ of~$X$ is in the center
of the automorphism group, $\sigma$ induces an automorphism of
order~$3$ of~$\P^1_{\Q}$. We can take this automorphism to be given
by $x \mapsto \frac{1}{1-x}$. Then the curve~$X$ has a model of the form
\begin{equation} \label{eqn:curveX}
  y^2 = r F_1(x,z)^2 + s F_1(x,z) F_2(x,z) + t F_2(x,z)^2
\end{equation}
with $r, s, t \in \Z$ and
\[ F_1(x,z) = x z (x-z) \qquad\text{and}\qquad F_2(x,z) = x^3 - x^2 z - 2 x z^2 + z^3 . \]
The action of~$\sigma$ is $(x : y : z) \mapsto (z : y : z-x)$;
$\sigma$ has four fixed points (satisfying $x^2 - x z + z^2 = 0$),
so the quotient curve $X/\langle \sigma \rangle$ has genus~$0$.
This implies that all divisors of the form $P + \sigma(P) + \sigma^2(P)$
(where $P$ is a point on~$X$) are linearly equivalent.
This means that $\sigma^2 + \sigma + 1 = 0$ in $\End(J)$,
so that $\Z[\omega] \subseteq \End_\Q(J)$, where $\omega$ is a primitive
cube root of unity. In particular, if $J$ is simple, then $J$ is
an abelian surface of $\GL_2$-type and therefore occurs as a simple
factor of~$J_1(N)$ for some~$N$ such that the conductor of~$J$ is~$N^2$;
see~\cites{ribet,khare-wintenberger1,khare-wintenberger2}.
Since the endomorphism algebra contains the CM~field~$\Q(\omega)$, the associated character
must be nontrivial. So if we can arrange for~$J$ to have conductor~$p^2$
for some prime~$p$ and to have positive (analytic) rank, then the
associated pair of newforms for~$\Gamma_1(p)$ will be strange.

The discriminant of the model of~$X$ given in~\eqref{eqn:curveX} above is
\[ \Delta(r,s,t)
    = 2^8 (s^2 - 4 r t)^3 \bigl(\tfrac{1}{4}((2r + s - 13t)^2 + 27 (s + t)^2)\bigr)^2 .
\]
If the right hand side in~\eqref{eqn:curveX} is a square modulo~$4$,
then we can ``un-complete the square'' and get a new model that is
still integral and whose discriminant is $\Delta(r,s,t)/2^{20}\!$.
If we assume that not all of $r, s, t$ are divisible by~$4$, then
we are in this case exactly when
\[  (r, s, t) \equiv (0, 0, 1), \quad (1, 0, 0) \quad\text{or}\quad (1, 2, 1) \bmod 4 . \]
In this case, $2^4 \mid s^2 - 4 r t$ (and the last factor in the
expression for~$\Delta(r,s,t)$ above is an integer).

Let $p \equiv 1 \bmod 6$ be a prime.
To get a curve~$X$ with (minimal) discriminant $\pm p^2$, we can
set
\[ s^2 - 4 r t = \pm 2^4 \qquad\text{and}\qquad
   (2r + s - 13t)^2 + 27 (s + t)^2 = 4 p .
\]
Then, up to perhaps a common sign change, $(r,s,t)$ satisfy one
of the congruences above, and we do obtain a curve with
discriminant~$\pm p^2$. Since $p \equiv 1 \mod 3$, we can always
write $4p = u^2 + 27 v^2$, and $u$ and~$v$ are uniquely determined
up to sign. Expressing $r$ and~$s$ in terms of $u$, $v$ and~$t$,
the first equation becomes
\[ 27 t^2 + 2 u t - v^2 \pm 2^4 = 0 . \]
The discriminant of this quadratic equation in~$t$ is
\[ 4(u^2 + 27 v^2 \mp 2^4 \cdot 27) = 4^2 (p \mp 108) . \]
So if there is to be a solution, $p \mp 108 = m^2$ must be a square.
The solutions are then
\[ t = \frac{-u \pm 2 m}{27} , \]
and since $u^2 + 27 v^2 = 4 p = 4 m^2 \pm 432$, we have that
$u \equiv \pm 2m \mod 27$, so that one of the two possibilities
will lead to an integral solution. We summarize the discussion
so far. Note that the conductor of~$J$ must be~$p^2$ if the
discriminant of~$X$ is $\pm p^2$, since the conductor must
be a square and it divides the discriminant.

\begin{proposition}
  Let $m \in \Z$ be such that $p = m^2 \pm 108$ is a prime. Write
  $4p = u^2 + 27 v^2$ with $u, v \in \Z$ such that $u \equiv 2m \bmod 27$
  and set $t = (2m-u)/27 \in \Z$. Then the curve~$X$ in~\eqref{eqn:curveX}
  with $(r,s,t)$ or $(-r,-s,-t)$, where $r = (u-v)/2 + 7t$ and $s = v-t$, the sign
  chosen so that $r \equiv 1 \bmod 4$ or $t \equiv 1 \bmod 4$, has minimal
  discriminant~$\pm p^2$. If its Jacobian~$J$ is simple, it occurs
  as a simple factor of~$J_1(p)$ with nontrivial character.
\end{proposition}

In the case $p = m^2 + 108 = m^2 + 27 \cdot 2^2$, we must have
$u = 2m$ and $v = \pm 4$, because $u$ and~$v$ are essentially unique.
We then obtain $t = 0$ and $r = \pm m - 2$, $s = 4$, with the sign chosen
so that $r \equiv 1 \bmod 4$. (The alternative $r = \pm m + 2$, $s = -4$
leads to an isomorphic curve). The right hand side splits
as a product of $x z (x-z)$ and a cubic (with cyclic Galois group).
Since $u$ and~$v$ are even, $2$ is a cubic residue mod~$p$ by a famous
result due to Gauss.

In the other case, $t \neq 0$, since $u = 2m$ would force $v^2$
to be negative.

We note that there are two primes, $p = 733$ and $p = 2917$ that can
be written as $m^2 + 108$ and as $m^2 - 108$ (with different~$m$).
For these primes, we obtain two non-isomorphic curves.

The list of all primes $p < 2 \cdot 10^5$ such that $p = m^2 + 108$ is as follows.
\begin{gather*}
  109, 157, 229, 277, 397, 733, 1069, 1789, 2917, 4597, 5437, 6037, 6997, 7333,
  8389, \\
  9133, 15733, 19429, 24133, 26029, 27997, 28669, 32869, 37357, 38917, 39709, \\
  43789, 51637, 55333, 58189, 60133, 67189, 72469, 76837, 87133, 90709, 93133, 101869, \\
  105733, 117757, 121909, 124717, 142237, \mathbf{151429}, 152989, 162517, 185869, 187597  .
\end{gather*}
The list of all primes $p < 2 \cdot 10^5$ such that $p = m^2 - 108$ is as follows.
\begin{gather*}
  13, \mathbf{61}, \mathbf{181}, \mathbf{421}, \mathbf{733}, \mathbf{853},
  1117, 1741, 2293, \mathbf{2917}, 3373, 3613, 4933,
  \mathbf{5821}, \mathbf{6133}, \\
  \mathbf{6781}, \mathbf{10093}, 10501, \mathbf{14533}, \mathbf{17053}, 17581, 18661,
  19213, \mathbf{20341}, \mathbf{22093}, 23917, \\
  30517, 32653, \mathbf{36373}, 43573, \mathbf{48733}, 51421, 54181, 55117, 57973, 60901,
  \mathbf{63901}, \\
  \mathbf{66973}, \mathbf{69061}, 70117, 72253, 78853, 82261, 89293, 97861, 101653, \mathbf{105517}, \mathbf{109453}, \\
  117541, \mathbf{133117}, 134581, \mathbf{139021}, 146581, \mathbf{151213}, 165541, 167173, 175453, \mathbf{192613} .
\end{gather*}
The numbers in boldface are those for which the associated Jacobian
has positive rank. In each of these cases, there is a pair of strange
newforms; these newforms all have character of order~$6$. Conversely,
all strange newform orbits of size~$2$ for a character of order~$6$
within our table are associated to curves~$X$ of the type considered here.
Somewhat surprisingly, only one of the primes of the form $m^2 + 108$ in this range
leads to a Jacobian of positive rank. (Amusingly, it has a strange ``twin''
of the form $m^2 - 108$ with the same $m = 389$.)

If we assume that the ``probability'' that a number of the form
$N = m^2 - 108$ is prime is a constant multiple of $1/\log N$ (which
seems reasonable), then we expect the sum over $\log p$ for such
primes $p < B$ to grow like a constant times~$\sqrt{B}$.
If we assume that all strange characters of order~$6$ are obtained
in this way, then the corresponding sum over the associated primes
will grow at most like a constant times~$\sqrt{B}$.

\begin{question}
  Is it true that \emph{every} pair of strange newforms with character of order~$6$
  is associated to the Jacobian of a curve of the form above?
  (Note that there are strange newforms with character of order~$6$
  at primes $1861$ and~$15\,061$ (which are not of the form $m^2 \pm 108$),
  but they come in Galois orbits of sizes $6$ and~$4$, respectively.)
\end{question}

\begin{question}
  Why is it rare for a curve as above associated to a prime $p = m^2 + 108$
  to have Jacobian with positive Mordell-Weil rank?
\end{question}

\begin{question}
  Are there similar explanations for the other cases that seem to occur
  relatively frequently? This refers to pairs of newforms with character of
  order~$3$ and to quadruples of newforms with character of order
  $5$ or~$12$ (or possibly~$4$).
\end{question}

We will come back to these questions in a follow-up paper.
It appears that there may be geometric constructions similar to the one
considered here related to newform spaces of certain dimensions and
character orders. For example, two-dimensional newform spaces at prime level~$p$
associated to cubic characters should correspond to (isogeny classes of)
elliptic curves~$E$ over~$\Q$ of conductor~$p^2$ and with reduction type IV
or~IV$^*$ at~$p$. Assuming BSD, such a newform space is strange if and only
if $E$ gains rank after a base-change to the cyclic cubic field of conductor~$p$.

Assuming Bunyakovsky's conjecture~\cite{bunyakovsky}*{No.~10}, we can at least show that there
are infinitely many strange primes of the type considered above.

\begin{proposition} \label[proposition]{prop:strange_infinite}
  Let $w$ be a positive integer. If
  \begin{equation} \label{eqn:wpol}
    p = w^4 + 2 w^3 + 23 w^2 + 22 w + 13 = (w^2 + w + 11)^2 - 108
  \end{equation}
  is a prime, then there is a pair of strange characters of order~$6$
  at level~$p$; in particular, $p$~is strange.
\end{proposition}

\begin{proof}
  We set
  \[ u = 2w^2 + 2w - 5, \quad v = 2w + 1, \quad t = 1, \]
  so
  \[ r = w^2 + 4, \quad s = 2w, \quad t = 1 . \]
  Then the curve
  \[ X_p \colon y^2 = r F_1(x,z)^2 + s F_1(x,z) F_2(x,z) + t F_2(x,z)^2 \]
  is of the type above for the prime~$p$, and it contains the rational
  point $P = (0 : 1 : 1)$. Denoting by~$\iota$ the hyperelliptic involution,
  we show that $P - \iota(P)$ represents a point~$Q$ of infinite order
  on the Jacobian~$J_p$ of~$X_p$. Reducing mod~$3$ and mod~$5$, we
  find that the point has order~$19$ mod~$3$ when $3 \mid w^2 + w$
  and order~$3$ otherwise, and it has order~$19$ mod~$5$ when
  $5 \mid w^2 + w$ and order $6$ or~$9$ otherwise.
  If $15$ does not divide~$w^2+w$, this implies that the orders
  of the reduced points differ. Since the reduction modulo an odd prime
  is injective on the torsion subgroup, these orders must agree when
  the point has finite order in~$J_p(\Q)$. This shows that the point
  must have infinite order. If $15 \mid w^2 + w$ and $Q$ has finite order,
  it must have order~$19$. We can construct the curve above with $w$ an
  indeterminate and compute $19$~times the image of~$Q$ on
  the associated Kummer surface. If this is to be the origin,
  the first three of the four coordinates must vanish simultaneously,
  which results in a polynomial equation that $w$ has to satisfy; it turns
  out that the only such values are $w \in \{-1, 0\}$,
  which are excluded by our assumptions.

  So $J_p(\Q)$ has positive rank. All possible reductions mod~$11$
  of~$X_p$ (for $w \not\equiv 5 \bmod 11$; otherwise
  the expression for~$p$ is divisible by~$11$) have a zeta function
  with irreducible numerator, so $J_p$ must be simple over~$\Q$.
  Since $J_p$ is a simple factor of~$J_1(p)$ of positive rank, by
  \cite{kolyvagin-logachev} and~\cite{kato}, it follows that the
  two associated newforms have positive analytic rank.
  Considering the Euler factor at~$2$ (it only depends on $w \bmod 4$ and is
  in fact independent of~$w$) shows that the associated character has order~$6$:
  a priori, its order has to divide~$6$ since it takes values in
  $\End^0_{\Q}(J_p) = \Q(\sigma)$; from the factorization of the Euler factor
  over~$\Q(\sigma)$ we see that $\chi(2)$ is a primitive sixth root of unity.
  This shows that the characters of order~$6$ mod~$p$ are strange.
\end{proof}

\begin{remark}
  Note that for $w \in \{-1, 0\}$, we have $p = 13$, and the
  resulting curve is~$X_1(13)$, whose Jacobian has rank zero
  and rational torsion of order~$19$.
\end{remark}

We note that \cref{prop:strange_infinite} explains the strange primes
\[ p = 61, \; 181, \; 421, \; 853, \; 6781, \; 10093, \; 14533, \; 20341, \; 48733 \]
in our list (where $w = 1, 2, 3, 4, 8, 9, 10, 11, 14$, respectively).
The next one is $p = 223\,621$ for $w = 21$, slightly outside our table.

Assuming Bunyakovsky's conjecture for the polynomial in~$w$
in~\eqref{eqn:wpol}, \cref{prop:strange_infinite} implies that
there are infinitely many strange primes.

To check the computational claims made in this section, run Magma with
\cite{code}*{\texttt{section6.magma}}.


\section{Behavior of $S(d)$ for large $d$} \label{Sect:appl_strange}

We now consider what we can say about~$S(d)$ when $d$ gets large.

Since it is more convenient to work with points of exact degree~$d$ on~$X_1(p)$,
we make the following definition.

\begin{definition}
  Let $d \ge 1$ be an integer. Then $\Snew(d)$ denotes the subset of~$S(d)$
  consisting of primes~$p$ such that $X_1(p)$ has a non-cuspidal point of
  exact degree~$d$.
\end{definition}

In other words, $p \in \Snew(d)$ means that there is a number field~$K$ of degree~$d$
and $x \in X_1(p)(K)$ such that for all strict subfields $L \subsetneq K$ one has
that $x$ is not in~$X_1(p)(L)$.
Since the modular curve $X_1(p)$ is a fine moduli space, this can also be stated
in terms of elliptic curves: $p \in \Snew(d)$ means that there is a pair~$(E, P)$
defined over a number field~$K$ of degree~$d$ such that $(E, P)$ is not $\Qbar$-isomorphic
to a pair~$(E', P')$ over a strict subfield $L \subsetneq K$.

\begin{remark}
  We have $S(d) = \bigcup_{d' \mid d} \Snew(d')$, where the union is over
  all (positive) divisors~$d'$ of~$d$.
\end{remark}

We have the following consequence of~\cref{prop:hecke_gon}.

\begin{corollary} \label[corollary]{cor:p_fixed}
  Let $p$ be an odd prime.
  Let $H$ be the subgroup of~$(\Z/p\Z)^\times/\{\pm 1\}$ that is
  the intersection of the kernels of all strange characters for~$p$
  (then the index of~$H$ is~$\str(p)$). Pick a generator~$a$ of~$H$
  and let $n$ be as in~\cref{prop:hecke_gon} for~$a$ and~$X_1(p)$
  (with $\ell = 2$ if the condition there is satisfied, else $\ell = 3$).
  Let $d \ge 1$ be an integer.
  If $p \in \Snew(d)$, then one of the following holds.
  \begin{enumerate}[\upshape(a)]
    \item \label{cor:p_fixed_a}
          $d \ge \gon_{\Q}(X_1(p))/n$.
    \item \label{cor:p_fixed_b}
          $p \equiv 1 \bmod 6$ and $d = (p-1)/3$.
    \item \label{cor:p_fixed_c}
          $p \equiv 1 \bmod 4$ and $d = (p-1)/2$.
    \item \label{cor:p_fixed_d}
          $d$ is a multiple of $\#H = (p-1)/(2 \str(p))$, say $d = m \cdot \#H$,
          and there is a non-cuspidal point of degree~$m$ on~$X_H$.
  \end{enumerate}
\end{corollary}

\begin{proof}
  Since $p \in \Snew(d)$, there is a
  point $x \in X_1(p)^{(d)}(\Q)$ coming from a non-cuspidal degree~$d$ point~$y$ on~$X_1(p)$.
  We assume that $n d < \gon_{\Q}(X_1(p))$, so that we are not
  in case~\eqref{cor:p_fixed_a}.
  Then $A = \bigl(\diamondop{a} - 1\bigr) \bigl(J_1(p)(\Q)\bigr)$ is finite:
  by~\cref{lem:simple_factors_of_JH} and the definition of~$H$,
  every simple factor of~$J_1(p)$ of positive analytic rank is a factor
  of~$J_H$; now use~\cites{kolyvagin-logachev,kato} as in the proof
  of~\cref{cor:medium_gen}. So by~\cref{prop:hecke_gon}, $x$ is a sum of orbits of~$H$.
  Since orbit length is stable under the action of the Galois group,
  all orbits occurring in~$x$ have the same length. This length can
  be $\#H$, $\#H/2$, or $\#H/3$; compare~\cref{rmk:orbits}.
  \begin{enumerate}[(i)]
    \item If the orbit length is~$\#H$, then $d = m \cdot \#H$ for some
          integer $m \ge 1$.
          Each $H$-orbit in the support of~$x$ is the pull-back
          of a point of degree~$m$ on~$X_H$, so we
          are in case~\eqref{cor:p_fixed_d}.
    \item If the orbit length is~$\#H/2$, then $p \equiv 1 \bmod 4$,
          and $y$ maps to a point on~$X_0(p)$ corresponding to
          an elliptic curve~$E$ with $j$-invariant~$1728$ together with
          one of the two subgroups of~$E$ of order~$p$ stable under~$\Z[i]$.
          Such a point has field of definition~$\Q(i)$, which is the
          CM-field of~$E$. So $\Q(i)$ is contained in the field of
          definition of~$y$; in particular, $d$ is even.
          By~\cite{silverberg}, it follows that $p-1 \le 2d$.
          On the other hand, there are exactly $(p-1)/2$ points
          on~$X_1(p)$ above the two relevant points on~$X_0(p)$
          ($(p-1)/4$ above each of the two; the covering is ramified
          with index~$2$ at these points). This implies that
          $d = (p-1)/2$, and we are in case~\eqref{cor:p_fixed_c}.
    \item If the orbit length is~$\#H/3$, then $p \equiv 1 \bmod 6$,
          and we can argue as in the previous case with $j = 0$
          in place of $j = 1728$ and $\Z[\omega]$ (and~$\Q(\omega)$)
          in place of~$\Z[i]$ (and~$\Q(i)$) to deduce that
          we must be in case~\eqref{cor:p_fixed_b}.
    \qedhere
  \end{enumerate}
\end{proof}

At the cost of a worse bound on~$d$, we can work with a generator~$a$
of~$(\Z/p\Z)^\times/\set{\pm 1}$ even when $p$ is a strange prime.
The bound depends on the strangeness dimension of~$p$, $\sdim(p)$.

\begin{corollary} \label[corollary]{cor:p_fixed_var}
  Let $p > 3$ be an odd prime.
  Pick a generator~$a$ of~$(\Z/p\Z)^\times/\set{\pm 1}$
  and let $n$ be as in~\cref{prop:hecke_gon} for~$a$ and~$X_1(p)$
  (with $\ell = 2$ if the condition there is satisfied, else $\ell = 3$).
  Let $d \ge 1$ be an integer.
  If $p \in \Snew(d)$, then one of the following holds.
  \begin{enumerate}[\upshape(a)]
    \item \label{cor:p_fixed_var_a}
          $d \ge \dfrac{\gon_{\Q}(X_1(p))}{n \floor{(2\sqrt{2} + 3)^{\sdim(p)}}}$.
    \item \label{cor:p_fixed_var_b}
          $p \equiv 1 \bmod 6$ and $d = (p-1)/3$.
    \item \label{cor:p_fixed_var_c}
          $p \equiv 1 \bmod 4$ and $d = (p-1)/2$.
    \item \label{cor:p_fixed_var_d}
          $d$ is a multiple of $(p-1)/2$, say $d = m (p-1)/2$,
          and there is a non-cuspidal point of degree~$m$ on~$X_0(p)$.
  \end{enumerate}
\end{corollary}

\begin{proof}
  Let $f_1, \ldots, f_m$, with $m = \sdim(p)$, be the strange newforms
  at level~$p$, and set $h = \prod_{j=1}^m (x - a_2(f_j)) \in \Z[x]$
  (compare~\cref{cor:charpoly_integral}).
  This polynomial has roots bounded in absolute value by~$2\sqrt{2}$,
  so the coefficient~$h_j$ of~$x^j$ is bounded by
  $\binom{m}{j} (2\sqrt{2})^{m-j}$.
  Similarly to our previous considerations in the proof of~\cref{prop:hecke_gon},
  it follows that
  \[ t = (T_\ell - \diamondop{\ell} - \ell) h(T_2) (\diamondop{a} - 1) \]
  annihilates~$J_1(p)(\Q)$ (with $\ell = 2$ or~$3$): $\diamondop{a} - 1$
  kills~$J_0(p)$ and $h(T_2)$ kills the part of~$J_1(p)$ corresponding to
  strange newforms, so that $h(T_2) (\diamondop{a} - 1) J_1(p)$ is an abelian
  variety of Mordell-Weil rank zero (again by~\cites{kolyvagin-logachev,kato}).
   In particular, $h(T_2) (\diamondop{a} - 1) J_1(p)(\Q)$
  is torsion and contained in $h(T_2) A \subseteq A$
  (where $A = \bigl(\diamondop{a} - 1\bigr) \bigl(J_1(p)(\Q)\bigr)$)
  and is therefore killed by~$T_\ell - \diamondop{\ell} - \ell$.

  We can write $t$ as a difference of two effective correspondences.
  Proceeding as in the proof of~\cref{prop:hecke_gon} and noting the bound on~$h_j$ from above,
  these correspondences can be taken to have degree at most
  \[ n \sum_{j=0}^m |h_j| \, 3^j
       \le n \floor{\sum_{j=0}^m \binom{m}{j} 3^j (2\sqrt{2})^{m-j}}
       = n \floor{(2 \sqrt{2} + 3)^m} \,.
  \]
  (In concrete cases, the degree can be smaller.)
  When we are not in case~\eqref{cor:p_fixed_var_a}, then it follows as
  in the proof of~\cref{prop:hecke_gon} that $x$ must be a sum of
  orbits of~$(\Z/p\Z)^\times/\{\pm 1\}$. (When $\ell \neq 2$, we need to
  apply~\cref{prop:hecke_op_kernel} twice in succession to deduce
  that $\diamondop{a} x = x$.)
  The remainder of the proof then is as for~\cref{cor:p_fixed}
  with $H$ the full group $(\Z/p\Z)^\times/\set{\pm 1}$.
\end{proof}

Fixing~$d$ instead of~$p$, we obtain the following.

\begin{corollary} \label[corollary]{cor:d+1}
  Let $d \ge 1$ be an integer. If $p \in \Snew(d)$ is a prime,
  then one of the following holds.
  \begin{enumerate}[\upshape(a)]
    \item \label{cor:d+1_a}
          $p \le \sqrt{Cd+1}$ with $C = 2^{19}/325 \approx 1613.2$.
    \item \label{cor:d+1_b}
          $d$ is even and $p = 2d+1$.
    \item \label{cor:d+1_c}
          $d$ is even and $p = 3d+1$.
    \item \label{cor:d+1_d}
          $p-1$ divides $2 \str(p) d$, and there
          are non-cuspidal points of degree $2 \str(p) d/(p-1)$
          on the intermediate curve~$X_H$ of $X_1(p) \to X_0(p)$
          with $\deg(X_H \to X_0(p)) = \str(p)$. If
          \begin{equation} \label{eq:p_str_ineq}
            p > \sqrt{C \floor{(2\sqrt{2} + 3)^{\sdim(p)}} d + 1} ,
          \end{equation}
          then $(p-1)/2$ divides~$d$, and there are non-cuspidal
          points of degree $2 d/(p-1)$ on~$X_0(p)$.
  \end{enumerate}
  Conversely, we have for any even~$d$ that $2d+1 \in \Snew(d)$ when
  $2d+1$ is prime and that $3d+1 \in \Snew(d)$ when $3d+1$ is prime.
\end{corollary}

\begin{proof}
  The first statement follows from~\cref{cor:p_fixed,cor:p_fixed_var},
  together with the gonality lower bound in~\cref{thm:gon_XH};
  note that we can always take $n = 8$ in~\cref{prop:hecke_gon}.

  For the second statement, note that we obtain a point of degree~$d$
  on~$X_1(p)$ for $p = 2d+1$ or $p = 3d+1$ when $p$ is prime and
  $d$ is even by taking a preimage in~$X_1(p)$ of one of the quadratic
  points on~$X_0(p)$ with $j$-invariant $1728$ or~$0$, respectively.
  (That the degree cannot be smaller follows from~\cite{CCS}*{Thm.~2(a)}, applied
  with $D = -4$ and~$D = -3$, respectively. Note that the field of
  definition of the point necessarily contains~$\Q(\sqrt{D})$.)
\end{proof}

\begin{remark} \label[remark]{rmk:163}
  If we are in case~\eqref{cor:d+1_d} in~\cref{cor:d+1} and $p$
  satisfies the bound~\eqref{eq:p_str_ineq}, then $p = \frac{2d}{m} + 1$
  for some $m \ge 1$ and there are non-cuspidal points of degree~$m$
  on~$X_0(p)$. Since $X_0(p)$ has no non-cuspidal rational points
  when $p > 163$ by~\cite{mazur2}, it follows that $m \ge 2$ in this case,
  and therefore $p \le d+1$.
\end{remark}

More generally, we have the following.

\begin{theorem} \label[theorem]{thm:Sd_gen_1}
  Assume~\cref{conj:str_unif}~\eqref{conj:str_unif_3}.
  Then for sufficiently large~$d$, we have that
  \[ \Snew(d) \subseteq \primes(d+1) \cup \set{2d+1, 3d+1} , \]
  and $2d+1 \in \Snew(d)$ (resp., $3d+1 \in \Snew(d)$) if and only if $d$
  is even and $2d+1$ is prime (resp., $3d+1$ is prime).
\end{theorem}

\begin{proof}
  By \cref{conj:str_unif}~\eqref{conj:str_unif_3},
  $p > \floor{(2\sqrt{2} + 3)^{\sdim(p)}}^2$
  when $p$ is sufficiently large. For such~$p$, we then have that
  the bound~\eqref{eq:p_str_ineq} is satisfied if
  $p > (Cd + 1)^{2/3}$. Take $d$ large enough so that
  $d + 1 > \max\set{163, (Cd + 1)^{2/3}}$. For such~$d$, \cref{cor:d+1}
  and~\cref{rmk:163} show that $p \in \Snew(d)$ implies that either
  $p \le (Cd + 1)^{2/3} < d+1$, or $p = (2d/m) + 1 \le d + 1$
  for some $m \ge 2$, or else $p \in \{2d+1, 3d+1\}$ if $d$ is even.
  \cref{cor:d+1} also gives us that $p \in \Snew(d)$ when $d$ is even
  and $p \in \{2d+1, 3d+1\}$.
\end{proof}

\begin{corollary}
  Assume~\cref{conj:str_unif}~\eqref{conj:str_unif_3}.
  \begin{enumerate}[\upshape(1)]
    \item For sufficiently large odd~$d$, we have that
          \[ S(d) \subseteq \primes(d+1) \,. \]
    \item For sufficiently large even~$d$, we have that
          \[ S(d) \subseteq \primes(d+1) \cup \set{\tfrac{3}{2}d + 1, 2d + 1, 3d + 1} \,. \]
  \end{enumerate}
\end{corollary}

\begin{proof}
  This follows from~\cref{thm:Sd_gen_1} together with $S(d) = \bigcup_{d' \mid d} \Snew(d')$.
\end{proof}

To obtain more precise results, we note that we can mimic for~$S_2(\Gamma_0(p))$
and the eigen\-spaces of the Atkin-Lehner involution what we
did in~\cref{sec:strange} using the splitting of~$S_2(\Gamma_1(p))$
according to characters. We do this in the next section.


\section{Small degree points on $X_0(p)$} \label{sec:X0}

Let $w$ denote the Fricke (=~Atkin-Lehner) involution on~$X_0(p)$.
Then the sign in the functional equation of $L(f, s)$, for $f$ a
weight~$2$ newform for~$\Gamma_0(p)$, is the negative of the eigenvalue
of~$w$ on~$f$. So the simple factors of~$J_0(p)$ on which $w$ acts
as multiplication by~$-1$ have analytic rank an even multiple of
their dimension (assuming that the analytic rank of a newform is invariant under
the Galois action, which is known when the analytic rank is $0$ or~$1$),
whereas the simple factors of the $w$-invariant
part~$J_0(p)^+$ (which is the Jacobian of
$X_0(p)^+ = X_0(p)/\langle w \rangle$) have analytic rank an odd
multiple of their dimension.

We now assume that there are no newforms that have positive
even analytic rank; equivalently, the minus part~$J_0(p)^-$ with
respect to the action of~$w$ has analytic rank zero and therefore
by~\cite{kolyvagin-logachev} Mordell-Weil rank zero, so that
$J_0(p)^-(\Q) \subseteq J_0(p)(\Q)_{\tors}$. Since the torsion
subgroup of~$J_0(p)(\Q)$ is cyclic and generated by the difference
of the two rational cusps~\cite{mazur1}*{Thm.~1}, it follows that $T_2 - 3$ annihilates
$J_0(p)(\Q)_{\tors}$ (since $T_2$ acts on the cusps as multiplication
by~$3$); therefore, $(T_2 - 3)(w - 1)$ annihilates $J_0(p)(\Q)$.
Let $x \in X_0(p)^{(d)}(\Q)$ be a non-cuspidal point and write
$\infty \in X_0(p)(\Q)$ for the cusp that is the image of
$\infty \in \P^1(\Q) \subseteq \HH^*$. Then (arguing in a similar way
as in the proof of~\cref{prop:hecke_gon}, but in a simplified situation)
$[x - d \cdot \infty] \in J_0(p)(\Q)$, so
\[ (T_2 - 3) (w - 1) ([x - d \cdot \infty])
     = (T_2 - 3) (w - 1) ([x]) = 0 \in J_0(p) .
\]
Writing
\[ (T_2 - 3) (w - 1) = (T_2 w + 3) - (T_2 + 3 w) \]
as a difference of two effective correspondences on~$X_0(p)$,
we have that $(T_2 w + 3)(x)$ is linearly equivalent to
$(T_2 + 3 w)(x)$. So if $6 d$, which is the degree
of these two divisors, is less than $\gon_{\Q}(X_0(p))$, then
it follows that the divisors must be the same. So we have that
\[ (T_2 - 3) (w - 1) (x) = 0 \]
as divisors. Now~\cref{prop:hecke_op_kernel}, applied to $t = T_2 - 3$
on~$X_0(p)$, shows that $w(x) = x$, since there are no cusps in the
support of~$x$. This leads to the following result.

\begin{proposition} \label[proposition]{prop:X0_w}
  Let $p > 3$ be a prime such that there are no newforms
  for~$\Gamma_0(p)$ that have positive even analytic rank.
  Let $d \ge 1$ be an integer. If $d < \gon_{\Q}(X_0(p))/6$
  and $y \in X_0(p)$ is a non-cuspidal point of degree~$d$,
  then either $y$ is a fixed point of the Fricke involution
  (and therefore corresponds to an elliptic curve with CM
  by an order of discriminant $-p$ or~$-4p$), or else $d$ is even
  and the image of~$y$ on~$X_0(p)^+$ is a point of degree~$d/2$.
\end{proposition}

\begin{proof}
  The discussion preceding the statement of the proposition shows
  that under the assumptions made, the Galois orbit of~$y$ is a
  union of $w$-orbits. Since $w$ is defined over~$\Q$, all these
  orbits have the same length. There are then two possibilities.
  \begin{enumerate}[(1)]
    \item The orbit length is $1$. Then $y$ is a fixed point of~$w$.
          If $(E, C)$ is the elliptic curve with a subgroup of order~$p$
          corresponding to~$y$, then this implies that there is an
          endomorphism~$\alpha$ of~$E$ such that $\alpha^2 = \pm p$
          (since $p \ge 5$, $\Aut(E) = \{\pm 1\}$). The positive sign
          is impossible, so $E$ must have CM by an order
          containing~$\sqrt{-p}$.
    \item The orbit length is $2$. Then $d$ is even, and the images
          on~$X_0(p)^+$ of the Galois conjugates of~$y$ correspond
          to the $d/2$ orbits under~$w$ among these points. This
          implies that the image of~$y$ on~$X_0(p)^+$ is a point
          of degree~$d/2$.
    \qedhere
  \end{enumerate}
\end{proof}

According to the \href{http://www.lmfdb.org/ModularForm/GL2/Q/holomorphic/?hst=List&level_type=prime&weight=2&char_order=1&analytic_rank=2-&search_type=List}{LMFDB},
exactly~$111$ of the $1229$ primes up to~$10\,000$ have the property that
there are weight~$2$ newforms for~$\Gamma_0(p)$ of positive even analytic rank.
They are listed in~\cref{table:bad_G0p}. We note that Brumer~\cite{brumer} has a
table listing the nontrivial splittings into Galois orbits of newforms with negative
Fricke eigenvalue; this table also gives what we denote $\dim(A)$ here as the
``number of rels''. Our table is consistent with his, except that we find two
additional positive-rank factors at levels $2333$ and~$2381$.

\begin{table}[htb]
\[ \renewcommand{\arraystretch}{1.2}
   \setlength{\doublerulesep}{5pt} 
  \begin{array}{|r|*{9}{c|}} \hline
    p          &  389 &  433 &  563 &  571 &  643 &  709 &  997 & 1061 & 1171 \\\hline
    \dim(A)    &    1 &    1 &    1 &    1 &    1 &    1 &    2 &    2 &    1 \\\hline
    \hline
    p          & 1483 & 1531 & 1567 & 1613 & 1621 & 1627 & 1693 & 1873 & 1907 \\\hline
    \dim(A)    &    1 &    1 &    3 &    1 &    1 &    1 &    3 &    1 &    1 \\\hline
    \hline
    p          & 1913 & 1933 & 2027 & 2029 & 2081 & 2089 & 2251 & 2293 & 2333 \\\hline
    \dim(A)    &    3 &    1 &    1 &    2 &    2 &    1 &    1 &    2 &    4 \\\hline
    \hline
    p          & 2381 & 2593 & 2609 & 2617 & 2677 & 2797 & 2837 & 2843 & 2861 \\\hline
    \dim(A)    &    2 &    4 &    2 &    2 &    1 &    1 &    1 &    4 &    2 \\\hline
    \hline
    p          & 2953 & 2963 & 3019 & 3089 & 3271 & 3463 & 3583 & 3701 & 3779 \\\hline
    \dim(A)    &    1 &    2 &    2 &    2 &    3 &    2 &    2 &    2 &    1 \\\hline
    \hline
    p          & 3911 & 3943 & 3967 & 4027 & 4093 & 4139 & 4217 & 4253 & 4357 \\\hline
    \dim(A)    &    2 &    4 &    1 &    2 &    2 &    1 &    2 &    3 &    1 \\\hline
    \hline
    p          & 4481 & 4483 & 4547 & 4787 & 4799 & 4951 & 5003 & 5171 & 5323 \\\hline
    \dim(A)    &    1 &    2 &    1 &    2 &    1 &    2 &    3 &    1 &    3 \\\hline
    \hline
    p          & 5351 & 5471 & 5477 & 5737 & 5741 & 5749 & 5813 & 5821 & 6007 \\\hline
    \dim(A)    &    2 &    3 &    2 &    3 &    1 &    2 &    1 &    2 &    2 \\\hline
    \hline
    p          & 6011 & 6043 & 6199 & 6337 & 6571 & 6581 & 6691 & 6949 & 7019 \\\hline
    \dim(A)    &    1 &    1 &    1 &    2 &    1 &    3 &    1 &    3 &    1 \\\hline
    \hline
    p          & 7451 & 7541 & 7621 & 7639 & 7669 & 7753 & 7867 & 7919 & 7933 \\\hline
    \dim(A)    &    1 &    1 &    2 &    2 &    1 &    3 &    1 &    2 &    2 \\\hline
    \hline
    p          & 8117 & 8219 & 8363 & 8369 & 8443 & 8513 & 8699 & 8747 & 9011 \\\hline
    \dim(A)    &    3 &    1 &    1 &    2 &    1 &    3 &    1 &    1 &    4 \\\hline
    \hline
    p          & 9127 & 9161 & 9203 & 9277 & 9281 & 9467 & 9479 & 9781 & 9829 \\\hline
    \dim(A)    &    1 &    1 &    1 &    1 &    2 &    1 &    1 &    2 &    2 \\\hline
    \hline
    p          & 9857 & 9907 & 9967 &      &      &      &      &      &      \\\hline
    \dim(A)    &    3 &    2 &    2 &      &      &      &      &      &      \\\hline
  \end{array}
\]
\caption{Primes $p < 10\,000$ such that there exist weight~$2$ newforms
         for~$\Gamma_0(p)$ with positive even analytic rank
         and dimensions of the corresponding part of~$J_0(p)^-$.}
\label{table:bad_G0p}
\end{table}

In most cases, there is only one Galois orbit of newforms with positive
even analytic rank. The exceptions are $p = 997$ with two newforms defined
over~$\Q$, $p = 1913$ with one orbit of size~$1$ and one of size~$2$,
$p = 2843$ with one orbit of size~$1$ and one of size~$3$, and
$p = 9829$ with two newforms defined over~$\Q$. \cref{table:bad_G0p} also
gives the number of newforms of positive even analytic rank, which
(assuming BSD; for dimension~$1$, this is verified by the LMFDB data) is the
same as the dimension of the smallest abelian subvariety of~$J_0(p)^-$
whose group of rational points has the same rank as~$J_0(p)^-(\Q)$.

When there are newforms~$f$ with positive even analytic rank,
then we can use a polynomial~$h$ that has all $a_2(f)$ as roots and
work with $(T_2 - 3) h(T_2) (w - 1)$. This gives the following,
which is similar in spirit to~\cref{cor:p_fixed_var}.

\begin{proposition} \label[proposition]{prop:X0_w_gen}
  Given $N \ge 0$, there is a number $\eps_N > 0$ such that
  whenever $p > 3$ is a prime such that there are at most~$N$
  weight~$2$ newforms for~$\Gamma_0(p)$ with positive even
  analytic rank and $1 \le d \le \eps_N p$ is an integer, then
  any non-cuspidal point~$y$ of degree~$d$ on~$X_0(p)$ is either a fixed point
  of~$w$ (and thus a CM~point) or $d$ is even and $y$ maps to a point of degree~$d/2$
  on~$X_0(p)^+$. We can take
  \[ \eps_N = \frac{325}{3 \cdot 2^{16} \floor{(2\sqrt{2} + 3)^N}} . \]
\end{proposition}

\begin{proof}
  Let $p$ be a prime as in the statement, and let $f_1, \ldots, f_m$
  be the $m \le N$ weight~$2$ newforms for~$\Gamma_0(p)$ that have
  positive even analytic rank. As in the proof of~\cref{cor:p_fixed_var},
  the monic polynomial
  \[ h = \prod_{j=1}^m (x - a_2(f_j)) \in \Z[x] \]
  (compare~\cref{cor:charpoly_integral})
  of degree~$m$ has roots bounded in absolute value by~$2\sqrt{2}$,
  so the coefficient~$h_j$ of~$x^j$ is bounded by
  $\binom{m}{j} (2\sqrt{2})^{m-j}$.
  Similarly to our previous considerations, it follows that
  $t = (T_2 - 3) h(T_2) (w - 1)$ annihilates~$J_0(p)(\Q)$, and
  $t$ can be written as a difference of
  effective correspondences of degree at most $6 \floor{(2 \sqrt{2} + 3)^N}$.
  The argument in the proof of~\cref{prop:X0_w}
  (using~\cref{prop:hecke_op_kernel} with $(T_2 - 3) h(T_2)$)
  then shows the claim for $d$ such that
  \[ d < \frac{\gon_{\Q}(X_0(p))}{6 \floor{(2 \sqrt{2} + 3)^N}} . \]
  Together with the gonality bound from~\cref{thm:gon_XH},
  this gives the stated value for~$\eps_N$.
\end{proof}

Obviously, depending on the specific~$a_2(f_j)$, we may be able to
improve the bound on~$d$ in concrete cases.

\begin{definition}
  For a prime~$p$, write $\perdim(p)$ (``positive even rank dimension'')
  for the number of weight~$2$
  newforms~$f$ for~$\Gamma_0(p)$ such that $w(f) = -f$ and $L(f,1) = 0$.
\end{definition}

Similarly to~\cref{conj:str_unif}, we formulate the following conjecture.

\begin{conjecture} \label[conjecture]{conj:X0} \strut
  \begin{enumerate}[\upshape(1)]
    \item \label{conj:X0_1}
          (Weak form)
          \[ \lim_{p \to \infty} \frac{\perdim(p)}{\log p} = 0 \]
          as $p$ runs through all prime numbers.
    \item \label{conj:X0_2}
          (Strong form)
          The positive even rank dimension~$\perdim(p)$ is uniformly bounded
          as $p$ runs through all prime numbers.
  \end{enumerate}
\end{conjecture}

This would give a strong form of the answer ``yes'' to Problem~1 in~\cite{brumer}
and to the similar question Murty asks at the end of page~264 in~\cite{murty}.

Cowan~\cite{cowan} has computed all Galois orbits of newforms for~$\Gamma_0(p)$
of weight~$2$ and prime level $p < 2 \cdot 10^6$. The data for $p < 10^6$ is
\href{https://www.lmfdb.org/ModularForm/GL2/Q/holomorphic/?level_type=prime&level=10000-&weight=2&dim=1-6}{available}
as part of the~LMFDB. He finds that for all $10^4 < p < 10^6$, each of the two Atkin-Lehner
eigenspaces contains orbits of size at most~$6$, together with one large orbit.
There is only one orbit of size~$6$ (which has negative
Atkin-Lehner sign) for $p = 171\,713$, there are two orbits of size~$5$ (both with
negative Atkin-Lehner sign) for $p = 26777, 86161$, and ten of size~$4$, of which
six (for $p = 13681, 28057, 35977, 63607, 185599, 794111$) have negative Atkin-Lehner sign.
For each prime in this range, the total size of these small orbits with negative Atkin-Lehner sign
is at most~$7$. This extends earlier work by Martin~\cite{martin} (who considers
more general weights and levels) in the case of weight~$2$ and prime level.
Martin formulates a conjecture (Conjecture~A in loc.~cit.)
that essentially says (taking $r = 1$) that for 100\% of the primes, each of the
two Atkin-Lehner eigenspaces contains only one Galois orbit of newforms.
Assuming this, \cite{martin}*{Thm.~1} concludes that 100\% of all newforms of weight~$2$
for~$\Gamma_0(p)$ with $p$ prime that have root number~$+1$ have analytic rank zero,
implying that primes~$p$ with $\perdim(p) > 0$ are sparse. Cowan's data
may suggest that there are no further Galois orbits of size at least~$5$ and
smaller than (say) a quarter of the number of newforms in the relevant
Atkin-Lehner eigenspace, and that Galois orbits of sizes $3$ or~$4$ are rare.
The recent preprint~\cite{cowan-martin} by Cowan and Martin has some heuristics
supporting a statement along these lines.

There are some related results and conjectures.
A result of Iwaniec and Sarnak~\cite{iwaniec-sarnak}*{Theorem~6}
(with $k = 2$, $D = 1$ and $N$ running through the primes)
shows that for any $\varepsilon > 0$, there is some~$p_0(\varepsilon)$
such that for every prime $p > p_0(\varepsilon)$, the proportion of newforms of
positive analytic rank among those of even analytic rank is at most $\tfrac{1}{2} + \varepsilon$.
(Assuming GRH, \cite{ILS}*{(1.54)} improves this to $\tfrac{7}{16} + \varepsilon$.
The expectation via the ``Density Conjecture'' is that this can be reduced to~$\varepsilon$.)
If we combine this with a version of Maeda's Conjecture in this setting
(a fairly weak form of which is formulated in~\cite{brumer}*{Problem~4}), which
says that the newforms in each of the two Atkin-Lehner eigenspaces of cusp forms
of weight~$2$ for~$\Gamma_0(p)$ should form one large Galois orbit and only
very few small ones (note that this is supported very well by Cowan's data
mentioned above), we obtain that $\perdim(p)$ is bounded by the total size of the
small Galois orbits since the large one will make up significantly more than half
of the newforms, so by the Iwaniec-Sarnak result all its newforms will have analytic
rank zero.

Based on Cowan's data (as provided by the LMFDB) and assuming that the large
Galois orbit with Fricke sign~$-1$ has analytic rank zero (compare the discussion above),
we determined the primes~$p < 10^6$ such that $\perdim(p) > 4$. We found four such primes:
\begin{gather*}
  \perdim(86\,161) = 5\,, \qquad \perdim(89\,209) = 5\,, \\
  \perdim(133\,117) = 5\,, \qquad \perdim(171\,713) = 6\,.
\end{gather*}
This gives some (perhaps weak) support for the strong form of the conjecture above.

Assuming that small (say, less than a quarter the dimension of the ambient space) orbit sizes
are bounded and orbits of size $\ge 2$ are rare (so that the total size of small orbits
of size at least~$2$ is bounded), $\perdim(p)$ is determined, up to a bounded amount,
by the number of isogeny classes of elliptic curves of conductor~$p$ of positive even (analytic) rank.
According to~\cite{PPVW}, the number of isogeny classes of elliptic curves of rank
at least~$2$ and height~$\le X$ should be of the order of~$X^{19/24}$, and one can
expect that this should also hold for conductor~$\le X$. This would predict about
a constant times~$p^{-5/24}$ positive even rank elliptic curves of (prime) conductor~$p$
on average. Watkins~\cite{watkins}*{Heuristic~3.1} has the more precise asymptotics
$\sim X^{19/24} (\log X)^{3/5}$ (for curves with discriminant bounded by~$X$).
The LMFDB has complete data on elliptic curves of prime conductor up to
$300$~million. Counting positive even rank elliptic curves of prime conductors
between successive multiples of a~million shows good agreement with the Watkins heuristic
(the constant factor derived from the data decreases slowly by about a quarter
as the conductor grows to~$10^8$ and from there on varies within a fairly small interval,
with no clear tendency to grow or shrink).
This then suggests that there should be only finitely many primes~$p$ such that
there are five or more positive even rank elliptic curves of conductor~$p$.
Indeed, within the range of the LMFDB, there are only three such primes:
$p = 4\,297\,609$ and $p = 151\,141\,051$ with five such curves, and $p = 161\,137\,637$
with six. There are six primes with four such curves, and then $109$ with three
and $2224$ with two. We think that this is a fairly strong indication that even
the strong form of~\cref{conj:X0} might hold.

\begin{theorem} \label[theorem]{thm:X0}
  The weak form of \cref{conj:X0} implies the following.
  \begin{enumerate}[\upshape(i)]
    \item Fix an odd integer $d \ge 1$. There is $p_0(d)$ such that for
          all primes $p \ge p_0(d)$, there are no non-cuspidal points of
          degree~$d$ on~$X_0(p)$.
    \item Fix an even integer $d \ge 1$. There is $p_0(d)$ such that for
          all primes $p \ge p_0(d)$, every non-cuspidal point of
          degree~$d$ on~$X_0(p)$ maps to a point of degree~$d/2$ on~$X_0(p)^+$.
  \end{enumerate}
\end{theorem}

\begin{proof}
  Take $p_0(d)$ so that $\eps_{\perdim(p)} p > d$ for all $p \ge p_0(d)$.
  This is possible by assumption. (Indeed, the weaker hypothesis
  $\log p - \perdim(p) \log (2\sqrt{2} + 3) \to \infty$ as $p \to \infty$
  would be sufficient.) Then \cref{prop:X0_w_gen} shows that a point
  of degree~$d$ on~$X_0(p)$ is either a CM point associated to an order
  of discriminant $-p$ or~$-4p$, or else $d$ is even and the point
  maps to a point of degree~$d/2$ on~$X_0(p)^+$. If we increase $p_0(d)$
  if necessary so that the class numbers of the
  orders of discriminant $-p$ or~$-4p$ (when they exist) are larger than~$d$
  (which is possible, since these class numbers tend to infinity with~$p$),
  then the first possibility is excluded, since these points have degree
  equal to the class number.
\end{proof}

\begin{remark}
  For $d = 1$, it is a theorem due to Mazur~\cite{mazur2} that the set of
  primes~$p$ such that there is a non-cuspidal rational point on~$X_0(p)$
  is
  \[ \primes(19) \cup \set{37, 43, 67, 163} ; \]
  for $p \in \set{19, 43, 67, 163}$, these points are
  fixed points of the Fricke involution.
\end{remark}

To check the computational claims made in this section, run Magma with
\cite{code}*{\texttt{section8.magma}}.


\section{Stronger results on $S(d)$ for large~$d$}

We can obtain a stronger result than~\cref{thm:Sd_gen_1} if we also assume~\cref{conj:X0}.

\begin{theorem} \label[theorem]{thm:Sd_gen_2}
  Assume Conjectures~\ref{conj:str_unif}~\eqref{conj:str_unif_3} and~\ref{conj:X0}~\eqref{conj:X0_1}.
  Fix $m \ge 1$. Then there is $d_0(m)$ such that for integers $d \ge d_0(m)$,
  \begin{enumerate}[\upshape(1)]
    \item if $d$ is odd,
          \[ \Snew(d) \subseteq \primes\Bigl(\frac{d}{m}\Bigr) \,; \]
    \item if $d$ is even,
          \[ \Snew(d) \subseteq \primes\Bigl(\frac{d}{m}\Bigr)
                        \cup \set{\frac{d}{k} + 1 : 1 \le k \le m, k \mid d}
                        \cup \set{2d + 1, 3d + 1} \,.
          \]
  \end{enumerate}
  If $d = k(p-1)$ with $1 \le k \le m$ and $y \in X_1(p)(K)$ is
  non-cuspidal with $[K : \Q] = d$, then $K$ is a cyclic Galois extension
  of a subfield~$L$ such that $[L : \Q] = 2k$.
\end{theorem}

\begin{proof}
  In the situation of~\cref{thm:Sd_gen_1} and its proof, we take $d$
  large enough so that $(Cd + 1)^{2/3} \le d/m$. Then for $p > d/m$
  such that $p \notin \set{2d+1, 3d+1}$, we get that $p = (2d/\mu) + 1$
  for some $1 \le \mu \le 2m$, and there are non-cuspidal points of
  degree~$\mu$ on~$X_0(p)$. If $d$ is also sufficiently large so that
  for primes $p > d/m$, the conclusion of~\cref{thm:X0} holds for
  all odd degrees below~$2m$, then it follows that $\mu = 2k$ must be even,
  hence $p = (d/k) + 1$. If $d$ is odd, the number $(d/k) + 1$ is even,
  so cannot be a prime.

  The statement on the fields~$K$ of definition of degree~$d$ points
  on~$X_1(p)$ follows from the facts that $X_1(p) \to X_0(p)$ is a cyclic
  Galois cover over~$\Q$ and that $y$ maps to a point of degree~$2k$
  on~$X_0(p)$. (Note that since $p \notin \set{2d+1, 3d+1}$, the cover
  is unramified at the relevant point.)
\end{proof}

\begin{corollary} \label[corollary]{cor:Sd_bound1}
  Assume Conjectures~\ref{conj:str_unif}~\eqref{conj:str_unif_3} and~\ref{conj:X0}~\eqref{conj:X0_1}.
  Fix $m \ge 1$. Then there is $d_0(m)$ such that for integers $d \ge d_0(m)$,
  \begin{enumerate}[\upshape(1)]
    \item if $d$ is odd,
          \[ S(d) \subseteq \primes\Bigl(\frac{d}{m}\Bigr) \,; \]
    \item if $d$ is even,
          \begin{align*}
            S(d) \subseteq \primes\Bigl(\frac{d}{m}\Bigr)
                        &\cup \set{\frac{d}{k} + 1 : 1 \le k \le m, k \mid d} \\
                        &\cup \set{2\frac{d}{k} + 1, 3\frac{d}{k} + 1 : 1 \le k \le 3m, 2k \mid d} \,.
          \end{align*}
  \end{enumerate}
  In particular, we then have
  \[ \lim_{d \to \infty, \,\text{\upshape $d$ odd}} \frac{\max S(d)}{d} = 0 \qquad\text{and}\qquad
     \limsup_{d \to \infty, \,\text{\upshape $d$ even}} \frac{\max S(d)}{d} = 3 \,.
  \]
\end{corollary}

\begin{proof}
  The upper bound for~$S(d)$ follows from~\cref{thm:Sd_gen_2}
  and $S(d) = \bigcup_{d' \mid d} \Snew(d')$.

  That the limit is zero for odd~$d$ follows by taking $m$ arbitrarily large.

  The statement on the limit for even~$d$ follows from the upper bound and by observing that
  there are infinitely many primes $p \equiv 1 \bmod 6$, so for $d = (p-1)/3$,
  we have $\max S(d) = 3d+1$ if $p$ is sufficiently large.
\end{proof}

Assuming the stronger forms of our conjectures, we can be more precise.
Let $h(D)$ denote the class number of the quadratic order of discriminant~$D$.
We set
\[ H(d) \colonequals \bigl\{\text{$p$ prime} : h(-4p) (p-1) = 2d \text{\ or\ }
                     \bigl(p \equiv 3 \bmod 4 \text{\ and\ } h(-p) (p-1) = 2d\bigr) \bigr\} \,.
\]

\begin{remark}
  It is known that $h(-p), h(-4p) = p^{1/2 + o(1)}$, which implies that
  \[ H(d) \subseteq \primes(d^{2/3+o(1)}) \qquad \text{as $d \to \infty$} \]
  and that the elements of~$H(d)$ are bounded below by~$d^{2/3-o(1)}$.

  We also note that each prime $p \equiv 1 \bmod 4$ belongs to exactly
  one set~$H(d)$ (for $d = h(-4p) (p-1)/2$, which is even), each
  prime $p \equiv 7 \bmod 8$ also belongs to exactly one set~$H(d)$
  (for $d = h(-4p) (p-1)/2 = h(-p) (p-1)/2$, which is odd), and each
  prime $p \equiv 3 \bmod 8$ (with the exception of $p = 3$, which
  is in~$H(1)$ only) belongs to exactly two sets~$H(d)$
  (for $d = h(-p) (p-1)/2$ and $d = h(-4p) (p-1)/2 = 3 h(-p) (p-1)/2$,
  which are both odd). Since the associated degrees~$d$ are roughly of
  size~$p^{3/2}$, this implies that actually most sets~$H(d)$ are empty.
  For example, for $d \le 10\,000$, only $231$ sets~$H(d)$ are nonempty,
  of which $8$ have two elements and one has three elements (and there
  are no larger sets):
  \begin{gather*}
    H(33) = \set{ 23, 67 }\,, \quad
    H(315) = \set{ 127, 211 }\,, \quad
    H(1485) = \set{ 271, 331 }\,, \\
    H(1701) = \set{ 379, 487 }\,, \quad
    H(2625) = \set{ 251, 1051 }\,, \quad
    H(4095) = \set{ 631, 1171 }\,, \\
    H(7875) = \set{ 1051, 2251 }\,, \quad
    H(8415) = \set{ 991, 1123, 1531 }\,, \quad
    H(9009) = \set{ 859, 2003 }\,.
  \end{gather*}
\end{remark}

\begin{lemma} \label[lemma]{lem:CM}
  For all $d \ge 1$, $H(d) \subseteq \Snew(d)$.
\end{lemma}

\begin{proof}
  Let $p \in H(d)$. First assume that $h(-4p) (p-1) = 2d$.
  There are elliptic curves with CM by~$\Z[\sqrt{-p}]$ defined over a number field
  of degree~$h(-4p)$. These curves give rise to points of degree~$h(-4p)$ on~$X_0(p)$,
  which are fixed points of~$w_p$, since they have an endomorphism of degree~$p$
  (given by $\sqrt{-p}$) whose kernel is defined over their field of definition.
  Pulling back to~$X_1(p)$, this gives points of degree dividing $(p-1)/2 \cdot h(-4p) = d$.
  On the other hand, \cite{bourdon-clark-ii}*{Thms.~1.2 and~4.1} (where for the latter
  the `$M = 1$' version in~\cite{bourdon-clark-i}*{Thm.~7.2} is sufficient) show that the
  least degree of a $p$-torsion point on such an elliptic curve is~$d$ (note that
  the degree of the ring class field is~$h(-4p)$ and that $a = 1 \le m$ in our application
  of their Thm.~1.2).

  Now assume that $p \equiv 3 \bmod 4$ and $h(-p) (p-1) = 2d$. Similarly as above,
  there are elliptic curves with CM by~$\Z[(1+\sqrt{-p})/2]$ defined over a number field
  of degree~$h(-p)$, which give rise to points of degree~$h(-p)$ on~$X_0(p)$ as before.
  Pulling back to~$X_1(p)$, this gives points of degree dividing $(p-1)/2 \cdot h(-p) = d$,
  and the same results as above show that $d$ is in fact the minimal possible degree.

  In both cases, we obtain that $p \in \Snew(d)$.
\end{proof}

We also set
\[ D(d) \colonequals \bigl\{\text{$p$ prime} : p-1 \mid d \text{\ or\ }
                                               p = 2d+1 \text{\ or\ } p = 3d+1\bigr\} \,. \]

\begin{theorem} \label[theorem]{thm:Sd_gen_strong}
  Assume~\cref{conj:X0}~\eqref{conj:X0_2} and~\cref{conj:str_unif}~\eqref{conj:str_unif_2}.
  Then there are constants $C_2 > C_1 > 0$ such that for all $d \ge 1$,
  \begin{enumerate}[\upshape(1)]
    \item when $d$ is odd,
          \[ \primes(\sqrt{C_1 d + 1}) \cup H(d) \subseteq \Snew(d)
                \subseteq \primes(\sqrt{C_2 d + 1}) \cup H(d)
                \subseteq \primes(d^{2/3+o_d(1)}) \,;
          \]
    \item when $d$ is even,
          \[ \primes(\sqrt{C_1 d + 1}) \cup H(d) \subseteq \Snew(d)
               \subseteq \primes(\sqrt{C_2 d + 1}) \cup H(d) \cup D(d) \,.
          \]
  \end{enumerate}
\end{theorem}

\begin{proof}
  The lower bound follows from~\cref{prop:Sprime} (noting that the proof actually
  shows that $\primes(\sqrt{C_1 d + 1}) \subseteq \Snew(d)$) and~\cref{lem:CM}.

  For the upper bound, let $\alpha = \eps_N^{-1}$ with $\eps_N$ as in~\cref{prop:X0_w_gen},
  where $N$ is a uniform upper bound for~$\perdim(p)$, and let
  $\beta = C \bigl\lfloor (2\sqrt{2} + 3)^s \bigr\rfloor$ with $C$ as in~\cref{cor:d+1},
  where $s$ is a uniform upper bound for~$\sdim(p)$. Assume that $p \in \Snew(d)$.
  Set $C_2 = \max\{2 (168/167) \alpha, \beta, 163^2\}$ and assume that $p > \sqrt{C_2 d + 1}$;
  then $p \ge 167$. By~\cref{cor:d+1}, either $d$ is even and $p \in \{2d+1, 3d+1\}$
  (and so $p \in D(d)$), or else $p-1$ divides~$2d$ and there is a non-cuspidal point~$y$
  of degree $2d/(p-1)$ on~$X_0(p)$. Note that
  \[ \frac{2d}{p-1} < \frac{2 (p^2 - 1)}{C_2 (p - 1)} < \frac{p + 1}{\alpha} \le \eps_N p \,, \]
  so \cref{prop:X0_w_gen}, applied with $d \leftarrow 2d/(p-1)$, tells us that
  $y$ is a fixed point of~$w_p$, which implies that $p \in H(d)$, or $d$ is even
  and $y$ comes from a point of degree~$d/(p-1)$ on~$X_0(p)^+$. So in this latter
  case, $p-1$ divides~$d$, and we again have that $p \in D(d)$.
\end{proof}

To formulate the corresponding statement for~$S(d)$, we set
\[ \tilde{H}(d) \colonequals \bigcup_{d' \mid d} H(d') \qquad \text{and} \qquad
   \tilde{D}(d) \colonequals \bigcup_{d' \mid d} D(d') \,.
\]
The sets $\tilde{H}(d)$ can be quite a bit larger than~$H(d)$, for example,
\[ \tilde{H}(6300) = \set{3, 5, 7, 11, 13, 19, 29, 31, 37, 43, 61, 101, 127, 151, 181, 211, 421} \]
has $17$~elements.

\begin{corollary} \label[corollary]{cor:Sd_bound2}
  Assume~\cref{conj:X0}~\eqref{conj:X0_2} and~\cref{conj:str_unif}~\eqref{conj:str_unif_2}.
  Then there are constants $C_2 > C_1 > 0$ such that for all $d \ge 1$,
  \begin{enumerate}[\upshape(1)]
    \item when $d$ is odd,
          \[ \primes(\sqrt{C_1 d + 1}) \cup \tilde{H}(d) \subseteq S(d)
                \subseteq \primes(\sqrt{C_2 d + 1}) \cup \tilde{H}(d)
                \subseteq \primes(d^{2/3+o_d(1)}) \,;
          \]
    \item when $d$ is even,
          \[ \primes(\sqrt{C_1 d + 1}) \cup \tilde{H}(d) \subseteq S(d)
               \subseteq \primes(\sqrt{C_2 d + 1}) \cup \tilde{H}(d) \cup \tilde{D}(d) \,.
          \]
  \end{enumerate}
  In either case we have $\max S(d) \le 3d + 1$ for $d$  sufficiently large.
\end{corollary}

\begin{proof}
  This follows from~\cref{thm:Sd_gen_strong} and $S(d) = \bigcup_{d' \mid d} \Snew(d')$.
\end{proof}


\end{document}